\documentclass[11pt]{amsart} 
\usepackage{amsmath,amssymb}
\usepackage[dvips]{graphicx}
\usepackage{psfrag}
\title{An Ore-type theorem for perfect packings in graphs}
\author{Daniela K\"uhn, Deryk Osthus and Andrew Treglown}
\thanks {The authors were supported by the EPSRC, grant no.~EP/F008406/1.}
\date{June 2, 2009}
\def\COMMENT#1{} 
\def\TASK#1{}
\begin{document}
\maketitle
\def\noproof{{\unskip\nobreak\hfill\penalty50\hskip2em\hbox{}\nobreak\hfill%
         $\square$\parfillskip=0pt\finalhyphendemerits=0\par}\goodbreak}
\def\endproof{\noproof\bigskip}
\newdimen\margin   
\def\textno#1&#2\par{%
     \margin=\hsize
     \advance\margin by -4\parindent
            \setbox1=\hbox{\sl#1}%
     \ifdim\wd1 < \margin
        $$\box1\eqno#2$$%
     \else
        \bigbreak
        \hbox to \hsize{\indent$\vcenter{\advance\hsize by -3\parindent
        \sl\noindent#1}\hfil#2$}%
        \bigbreak
     \fi}
\def\proof{\removelastskip\penalty55\medskip\noindent{\bf Proof. }}
\def\C{\mathcal{C}}
\def\eps{\varepsilon}
\def\ex{\mathbb{E}}
\def\prob{\mathbb{P}}
\def\eul{{\rm e}}
\newtheorem{firstthm}{Proposition}
\newtheorem{thm}[firstthm]{Theorem}
\newtheorem{prop}[firstthm]{Proposition}
\newtheorem{lemma}[firstthm]{Lemma}
\newtheorem{cor}[firstthm]{Corollary}
\newtheorem{problem}[firstthm]{Problem}
\newtheorem{defin}[firstthm]{Definition}
\newtheorem{conj}[firstthm]{Conjecture}
\newtheorem{claim}[firstthm]{Claim}

\begin{abstract}
We say that a graph $G$ has a perfect $H$-packing (also called an $H$-factor) if
there exists a set of
disjoint copies of $H$ in $G$ which together cover all the vertices of $G$.
Given a graph $H$, we determine, asymptotically, the Ore-type degree condition which
 ensures that a graph
$G$ has a perfect $H$-packing. More precisely, let $\delta_{\rm Ore} (H,n)$ be the
smallest number $k$ such that every
graph $G$ whose order $n$ is divisible by $|H|$ and with $d(x)+d(y)\geq k$ for all
non-adjacent $x \not = y \in V(G)$ contains a perfect $H$-packing. We determine
$\lim_{n\to \infty} \delta_{\rm Ore} (H,n)/n$.
\end{abstract}

\section{Introduction}\label{sec1}
\subsection{Perfect packings in graphs of large minimum degree}
Given two graphs $H$ and $G$, an \emph{$H$-packing in $G$} is a collection of vertex-disjoint
copies of $H$ in~$G$. An $H$-packing is called \emph{perfect} if it covers all the
vertices of~$G$.
In this case one also says that $G$ contains an \emph{$H$-factor}. $H$-packings are
generalisations
of graph matchings (which correspond to the case when $H$ is a single edge).

In the case when $H$ is an edge, Tutte's theorem characterises those graphs which
have a
perfect $H$-packing. However, for other connected graphs $H$ no characterisation is
known.
Furthermore, Hell and Kirkpatrick~\cite{hell} showed that the decision problem
whether a graph~$G$ has a perfect
$H$-packing is NP-complete precisely when $H$ has a component consisting of at least
$3$ vertices.
It is natural therefore to ask for simple sufficient conditions which ensure the
existence of
a perfect $H$-packing. One such result is a theorem of Hajnal and Szemer\'edi~\cite{hs}
which states that a graph $G$ whose order $n$ is divisible by $r$ has a perfect
$K_r$-packing
provided that $\delta (G) \geq (1-1/r)n$. It is easy to see that the minimum degree
condition
here is best possible. So for $H=K_r$, the parameter which governs the existence of
a perfect
$H$-packing in a graph~$G$ of large minimum degree is $\chi (H)=r$.

The first two authors~\cite{kuhn, kuhn2} showed that for any graph $H$ either
the so-called critical chromatic number or the chromatic number of $H$
is the relevant parameter.
Here the \emph{critical chromatic number} $\chi _{cr} (H)$ of a graph $H$ is defined
as 
$$\chi _{cr} (H):=(\chi (H)-1)\frac{|H|}{|H|-\sigma (H)},$$ 
where $\sigma (H)$ denotes the size of the smallest possible 
colour class in any $\chi(H)$-colouring of $H$. Throughout the paper, we only
consider graphs~$H$
which contain at least one edge (without mentioning this explicitly), so
$\chi_{cr}(H)$ is well defined. 
Note that $\chi (H) -1 < \chi _{cr} (H) \leq \chi (H)$
for all graphs $H$, and $\chi _{cr} (H)=\chi (H)$ precisely when every $\chi
(H)$-colouring of $H$
has colour classes of equal size. The characterisation of when $\chi (H)$ or $\chi
_{cr} (H)$ is
the relevant parameter depends on the so-called highest common factor of $H$, which
is defined as follows.

We say that a colouring of $H$ is \emph{optimal} if it uses exactly $\chi (H)=:r$
colours. 
Given an optimal colouring $c$ of $H$, let $x_1 \leq x_2 \leq  \dots \leq x_r$ denote
the sizes of the colour classes of~$c$. We write $\mathcal D (c) := \{ x_{i+1} - x_i
\mid  i=1, \dots , r-1\}$,
and let $\mathcal D (H)$ denote the union of all the sets $\mathcal D (c)$ taken over
all optimal colourings $c$ of $H$. We denote by ${\rm hcf}_{\chi} (H)$ the 
highest common factor of all integers in $\mathcal D(H)$. If $\mathcal D (H) =\{ 0 \}$
then we define ${\rm hcf} _{\chi} (H) := \infty $. We write ${\rm hcf}_c(H)$ for the
highest
common factor of all the orders of components of $H$. For non-bipartite graphs $H$
we say that
${\rm hcf} (H)=1$ if ${\rm hcf} _{\chi} (H)=1$. If $\chi (H)=2$ then we say
${\rm hcf} (H)=1$ if ${\rm hcf} _c (H)=1$ and $ {\rm hcf }_{\chi} (H)\leq 2$.
(See~\cite{kuhn2} for some examples.) Put 
$$\chi ^* (H):=\begin{cases}
\chi _{cr} (H)& \text{if ${\rm hcf}(H)=1$;}\\
\chi (H)& \text{otherwise.}
\end{cases}$$
Also let $\delta (H,n)$ denote the smallest integer $k$ such that every graph $G$
whose order
$n$ is divisible by $|H|$ and with $\delta (G)\geq k$ contains a perfect $H$-packing.
\begin{thm}\label{kuhn} \cite{kuhn2} For every graph $H$ there exists a constant
$C=C(H)$ such that
$$\left(1-\frac{1}{\chi ^* (H)}\right)n-1 \leq \delta (H,n) \leq
\left(1-\frac{1}{\chi ^* (H)}\right)n+C.$$
\end{thm}
Theorem~\ref{kuhn} improved previous bounds by Alon and Yuster~\cite{ay}, who showed
that
$\delta (H,n)\le (1-1/\chi(H))n+o(n)$, and  by Koml\'os, S\'ark\"ozy and
Szemer\'edi~\cite{kss},
who replaced the $o(n)$-term  by a constant depending only on~$H$. Further related
results are discussed
in the surveys~\cite{kky, kom, k, survey, yuster}.

\subsection{Ore-type degree conditions for perfect packings}
Of course, one can also consider other types of degree conditions that ensure a
perfect $H$-packing
in a graph~$G$. One natural such condition is an Ore-type degree condition requiring
a lower bound on the sum of the degrees of non-adjacent vertices of~$G$. 
(The name comes from Ore's theorem~\cite{ore}, which states that a graph~$G$ of
order $n\ge 3$ contains a
Hamilton cycle if $d(x)+d(y) \geq n$ for all non-adjacent $x \not = y\in V(G)$.)

A result of Kierstead and Kostochka~\cite{kier} on equitable colourings implies that 
a graph $G$ whose order $n$ is divisible by $r$ and with $d(x)+d(y) \geq
2(1-1/r)n-1$ for all non-adjacent
$x\not =y \in V(G)$ contains a perfect $K_r$-packing. Note that this is a
strengthening of
the Hajnal-Szemer\'edi theorem. 
Kawarabayashi~\cite{kawa} asked for the Ore-type condition which guarantees a $K^-
_4$-packing
in a graph $G$ covering a given number of vertices of~$G$. (Here $K^-_4$ denotes the
graph
obtained from $K_4$ by removing an edge.) Similarly it is natural to seek an
Ore-type analogue 
of Theorem~\ref{kuhn}. This will be the main result of this paper (but with an
$o(n)$-error term).
Perhaps surprisingly, the Ore-type condition needed is not `twice the minimum degree
condition'.
For some graphs~$H$ it depends on the so-called colour extension number of~$H$,
which we will define now. Roughly speaking, this is a measure of how many extra
colours we
need to properly colour $H$ if we try to build this colouring by extending an
$(r-2)$-colouring
of a neighbourhood of a vertex of~$H$.

More precisely, suppose that~$H$ is a graph with $\chi (H)=:r$ which contains a
vertex~$x$
for which the subgraph $H[N(x)]$ induced by the neighbourhood of~$x$ is
$(r-2)$-colourable.
Given such a vertex $x \in V(H)$, let $m_x$ denote the smallest integer for which
there exists 
an $(r-2)$-colouring of $H[N(x)]$ that can be extended to an $(r+m_x)$-colouring of
$H$. 
The \emph{colour extension number $ CE (H)$ of $H$} is defined as
$$ CE (H):=  \min \{m_x \text{ } | \text{ }x \in V(H) \text{ with } \chi (H[N(x)])
\leq r-2\}.$$
If $\chi (H[N(x)])=r-1$ for all $x \in V(H)$ we define $CE (H):= \infty$.
So every bipartite graph $H$ without isolated vertices has $CE (H)=\infty$. All
other bipartite graphs $H$ have $CE(H)=0$.
In general, $1 \leq CE(H) < \infty$ if for any optimal colouring of~$H$ and any $v
\in V(H)$,
$N(v)$ lies in exactly $r-1$ colour classes of $H$, but there exists a vertex $x \in
V(H)$ such that
$\chi (H[N(x)])\leq r-2$. Note that in this case $CE (H) \leq r-2$. (Indeed,
we can colour $H-N(x)$ with $r$ different colours to obtain a $(2r-2)$-colouring
of~$H$.)

In order to help the readers to familiarize themselves with the notion of the colour
extension number we now give a number of examples. 
$\chi (K_4 ^-)=3$ and $\chi (K^- _4 [N(x)])=2$ for every vertex $x$ of $K^- _4$. Thus $CE(K^- _4)=\infty$.
Next consider the graph~$F^\diamond$ obtained from the complete $3$-partite graph $K_{2,2,2}$ by removing an edge
$xy$ of $K_{2,2,2}$ and adding a new vertex $z$ which is adjacent to $x$ and $y$ only.
Then $\chi (F^\diamond)=3$, $\chi (F^\diamond[N(w)])=2$ for every vertex $w \not = z$ in $F^\diamond$ and
$\chi(F^\diamond[N(z)])=1$. 
Note that in any $3$-colouring of $F^\diamond$, $x$ and $y$ are coloured differently. So if we $1$-colour
$N(z)=\{x,y\}$, this colouring can be extended to a $4$-colouring of $F^\diamond$ but not a $3$-colouring.
Thus $CE(F^\diamond)=1$.

For each~$k\ge 1$ and $r \ge k+2$ we now give an example%
      \COMMENT{By choosing the sizes of the colour classes~$V_i$ suitably we can
construct $H^\diamond$ with
${\rm hcf}(H^\diamond)=1$ and also~$H^\diamond$ with ${\rm hcf}(H^\diamond)\neq 1$.}
of a family of graphs~$H^\diamond$ with $CE (H^\diamond)=k$ and $\chi(H^\diamond)=r$.
Consider a complete $r$-partite graph whose vertex classes $V_1, \dots, V_r$
have size~$>k$. Let~$H^\diamond$ be obtained from this graph by deleting the edges of~$k$
vertex-disjoint
copies $K^1, \dots, K^{k}$ of $K_{k+1}$ which lie in $V_1 \cup \dots \cup V_{k+1}$, and
by adding a new vertex~$x$ which is adjacent to the $k(k+1)$ vertices lying in these
copies of~$K_{k+1}$ as well as to all the vertices in $V_{k+2},\dots,V_{r-1}$ (see
Figure~1).
\begin{figure}[htb!]
\begin{center}\small
\psfrag{x}[][]{$x$}
\psfrag{2}[][]{$V_1$}
\psfrag{3}[][]{$V_2$}
\psfrag{4}[][]{$V_3$}
\psfrag{5}[][]{$V_4$}
\psfrag{6}[][]{$V_5$}
\includegraphics[width=0.35\columnwidth]{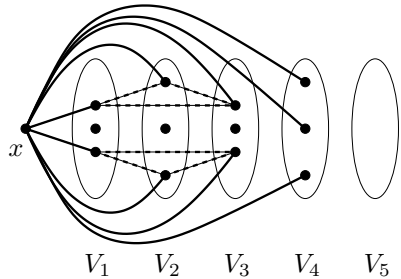}  
\caption{The graph~$H^\diamond$ in the case when $k=2$, $r=5$ and when each~$V_i$
has size~3. The dashed lines indicate
the deleted edges.}
\end{center}
\end{figure}
Note that $\chi (H^\diamond)=r$. Furthermore, any vertex $y \in V_1\cup\dots\cup V_r$
lies in a copy of $K_r$ in~$H^\diamond$. So $\chi (H^\diamond[N(y)])=r-1$.
However, the subgraph $D:=H^\diamond[N(x)\cap V_1\cap\dots\cap V_{k+1}]$ has a
$k$-colouring~$c'_x$ with colour classes $V(K^1), \dots, V(K^k)$
and it is easy to check that this is the only $k$-colouring of~$D$ (and so in
particular
$\chi(D)=k$).%
      \COMMENT{Indeed, we show by induction on~$k$ that if $D_k$ is obtained
from a complete $(k+1)$-partite graph with vertex classes $V_1,\dots,V_{k+1}$ of
size~$k$
by deleting $k$ vertex-disjoint copies $K^1,\dots,K^k$ of~$K_{k+1}$ then the only
$k$-colouring of $D_k$ is the one with colour classes $V(K^1),\dots,V(K^k)$.
This is easy if $k=1$. Suppose that $k\ge 2$ and let $x_i$ be the vertex of~$K^i$
in~$V_1$.
Consider any $k$-colouring~$c$ of $D_k$ and suppose that $x_1$ has colour~1. Then no
vertex in $D_k-(\{x_2,\dots,x_k\}\cup V(K_1))$ has colour~1 since each such vertex
is joined to~$x_1$.
Thus~$c$ induces a $(k-1)$-colouring of $D_k-(\{x_2,\dots,x_k\}\cup
V(K_1))=D_{k-1}$. But by
induction the colour classes of this colouring must be
$V(K^2)\setminus \{x_2\},\dots,V(K^k)\setminus \{x_k\}$. Wlog $V(K^i)\setminus
\{x_i\}$ has colour~$i$.
Consider the vertex $v\in V(K^1)\cap V_j$ ($j\ge 2$). Since $v$ is joined to all
the vertices in~$V(K^i)\setminus (V_j\cup V_1)\neq \emptyset$ for all $i\ge 2$,
this means that $v$ has neighbours of colours $2,\dots,k$ and thus $v$ has colour~1
in $c$.
(So $V(K^1)$ has colour~1.) Finally, each $x_i$ with $i\neq 1$ must have colour~$i$
since
it is joined to vertices of all other colours.}
Thus $\chi(H^\diamond[N(x)])=r-2$ and the only $(r-2)$-colouring of $H^\diamond[N(x)]$ is the one
which agrees with~$c'_x$
on~$D$ and colours each of $V_{k+2},\dots,V_{r-1}$ with a new colour. Let~$c_x$
denote this
colouring. When extending~$c_x$ to a proper colouring of~$H^\diamond$ we cannot reuse
the~$r-2$ colours used in~$c_x$
since every $y \in V(H^\diamond)\setminus N(x)$ is adjacent to a vertex in each colour class
of~$c_x$.
As $\chi (H^\diamond-N(x))=r-(r-k-2)=k+2$ this means that we require $r+k$ colours in total
to extend~$c_x$
to a proper colouring of~$H^\diamond$. Thus $CE (H^\diamond)=k$.

Let
$$
\chi _{\rm {Ore}} (H):=
\begin{cases}
\chi  (H)& \text{if }{\rm hcf}(H)\not =1 \text{ or }CE (H)=\infty ;\\
\max\left\{\chi_{cr} (H),\chi (H)-\frac{2}{CE(H)+2} \right\}& \text{otherwise.}
\end{cases}
$$
Recall that $CE(K^- _4)=\infty$ and $CE(F^\diamond)=1$, where $F^\diamond$ was defined above.
So $\chi _{\rm Ore} (K^- _4)=\chi (K^- _4)=3$. Any $3$-colouring
of $F^\diamond$ has one colour class of size $3$ and two colour classes of size $2$. So ${\rm hcf} (F^\diamond)=1$
and thus $\chi _{\rm Ore} (F^\diamond) = \max \{\chi_{cr} (F^\diamond),3-2/3 \}=\max \{ 14/5, 7/3\}=14/5.$

Note that if ${\rm hcf} (H)=1$ and $CE (H)=0$ then $\chi _{\rm {Ore}} (H)=\chi _{cr}
(H)$
(an odd cycle of length at least~$5$ provides an example of such a graph $H$).
On the other hand, one can choose the sizes of the vertex classes $V_i$ in the preceding example $H^\diamond$ so
that 
$\chi_{\rm Ore}(H^\diamond)$ lies strictly between $\chi_{cr}(H^\diamond)$ and $\chi(H^\diamond)$.
(For instance, take $k$ large, $|V_1|=k+1$, $|V_2|=2k$ and $|V_i| = 2k+1$ for all
$i\ge 3$.
Then $\chi_{cr}(H^\diamond)$ is close to $\chi(H^\diamond)-1/2$, ${\rm hcf}(H^\diamond)=1$ and so $\chi_{\rm
Ore}(H^\diamond)=\chi (H^\diamond)-2/(k+2)$.)

Given a graph $H$, let $\delta_{\rm Ore} (H,n)$ be the smallest integer $k$ such
that every
graph $G$ whose order $n$ is divisible by $|H|$ and with $d(x)+d(y)\geq k$ for all
non-adjacent
$x \not = y \in V(G)$ contains a perfect $H$-packing.
Roughly speaking, our next result states that when considering an Ore-type degree
condition, for any graph
$H$, $\chi _{\rm Ore} (H)$ is the relevant parameter which governs the existence of
a perfect $H$-packing.
In particular, it implies that we do not have a `dichotomy' involving only
$\chi(H)$ and $\chi_{cr}(H)$ as in Theorem~\ref{kuhn}. 
\begin{thm}\label{main}
For every graph $H$ and each $\eta>0$ there exists a constant $C=C(H)$
and an integer $n_0=n_0 (H,\eta)$ such that if $n \geq n_0$ 
then $$2\left(1-\frac{1}{ \chi_{\rm {Ore}}(H)}\right)n-C\le \delta_{\rm Ore} (H,n)
\le 2\left(1-\frac{1}{\chi_{\rm {Ore}}(H)}+\eta\right)n.$$
\end{thm}
So for example, Theorem~\ref{main} implies that
$\lim_{n\to \infty} \delta_{\rm Ore} (K^- _4,n)/n= 4/3$ and \\$\lim_{n\to \infty} \delta_{\rm Ore} (F^\diamond ,n)/n =9/7$.

The upper bound in Theorem~\ref{main} follows from Lemmas~\ref{generalH}
and~\ref{hcf} in 
Section~\ref{sec3}, which in turn are proved in Sections~\ref{sec3} and~\ref{sec6}.
The lower bound is proved in Section~\ref{sec2}.
For every graph $H$ there are infinitely many values of $n$ for which we can take
$C=2$  in
Theorem~\ref{main}. In fact, if ${\rm hcf }(H)\not =1$ or
$CE (H)=\infty$ then $C=2$ suffices for all $n$ divisible by $|H|$. In general $C
\leq 2|H|^4$
(see Section~\ref{sec2}).
It would be interesting to know whether one can replace the error term $\eta n$ 
by a constant depending only on $H$.

\subsection{Almost perfect packings}
The critical chromatic number was first introduced by Koml\'os~\cite{ko}, who showed
that 
it is the relevant parameter when considering `almost' perfect $H$-packings.
\begin{thm}\label{thmKomlos} \cite{ko}
For every graph $H$ and each $\gamma>0$ there exists an integer
$n_0=n_0(\gamma,H)$ such that every graph $G$ of order $n\ge n_0$
and minimum degree at least $(1-1/\chi_{cr}(H))n$ contains an $H$-packing
which covers all but at most $\gamma n$ vertices of~$G$.
\end{thm}
It is easy to see that the bound on the minimum degree in Theorem~\ref{thmKomlos} is
best possible.
In the proof of Theorem~\ref{main} we will use the following result which provides an 
Ore-type analogue of Theorem~\ref{thmKomlos}.
Again, the critical chromatic number is the relevant parameter for any graph~$H$.
In particular, this means that Theorem~\ref{orekomlos} is a generalization of
Theorem~\ref{thmKomlos}.
The proof of Theorem~\ref{orekomlos} is almost identical to that of
Theorem~\ref{thmKomlos}. A sketch of the proof is given in Section~\ref{sketch}. Full details
can be found in~\cite{treglown}.  
\begin{thm}\label{orekomlos}
For every graph~$H$ and each $\eta>0$ there exists an
integer $n_0=n_0 (H, \eta)$ such that if $G$ is a graph on $n\geq n_0$ vertices and
$$d(x)+d(y) \geq 2 \left( 1 - \frac{1}{\chi _{cr} (H) } \right)n$$
for all non-adjacent $x \not = y\in V(G)$ then $G$ has an $H$-packing covering all but
at most~$\eta n$ vertices. 
\end{thm}
Shokoufandeh and Zhao~\cite{szhao} showed that in Theorem~\ref{thmKomlos}
the bound on the number of uncovered vertices
can be reduced to a constant depending only on~$H$. We conjecture that this
should also be the case for~Theorem~\ref{orekomlos}.

\subsection{Copies of~$H$ covering a given vertex}
In the proof of Theorem~\ref{main} it will be useful to determine the Ore-type
degree condition
which guarantees a copy of~$H$ covering a given vertex of~$G$. 
Let $\delta_{\rm Ore}'(H,n)$ denote the smallest integer~$k$ such that
whenever~$w$ is a vertex of a graph~$G$ of order~$n$ with $d(x)+d(y) \geq k$ for all 
non-adjacent $x \not =y \in V(G)$ then~$G$ contains a copy of $H$ covering~$w$.
Define 
$$
\chi ' _{\rm {Ore}} (H):=\begin{cases}
\chi  (H)& \text{if }CE (H)=\infty ;\\
\chi (H)- \frac{2}{CE(H)+2}& \text{otherwise.}
\end{cases}
$$
\begin{thm}\label{oreHx}
For every graph $H$ and every $\eta >0$ there exists an integer $n_0 =n_0 (H, \eta)$
and a constant $C=C(H)$ such that if $n \geq n_0$ then
$$2\left(1-\frac{1}{ \chi '_{\rm {Ore}}(H)}\right)n-C\le \delta_{\rm Ore}'(H,n) \le
2\left(1-\frac{1}{\chi' _{\rm {Ore}}(H)}+\eta\right)n.$$
\end{thm}
Theorem~\ref{oreHx} is proved in Section~\ref{sec3}.
As in the case of perfect $H$-packings, the Ore-type degree condition in
Theorem~\ref{oreHx}
does not quite match the bound needed for the corresponding minimum degree version.
Indeed, let $\delta ' (H,n)$ denote the smallest integer $k$ such that whenever~$w$
is a vertex
of a graph~$G$ of order~$n$ with $\delta (G) \geq k$ then~$G$ contains a copy of $H$
covering~$w$.
Together with the Erd\H{o}s-Stone theorem the next result implies that asymptotically
$\delta ' (H,n)$ is the same as the minimum degree needed to force any copy of~$H$
in a graph of order~$n$. 
\begin{prop}\label{minHx}
For every graph $H$ and every $\eta >0$ there exists an integer
$n_0=n_0 (H, \eta)$ such that if $n \geq n_0$ 
then
$$\left(1-\frac{1}{\chi (H)-1}\right)n-1 \leq \delta'(H,n) \leq
\left(1-\frac{1}{\chi (H)-1}+\eta \right)n .$$
\end{prop}
The lower bound on $\delta '(H,n)$ in Proposition~\ref{minHx} follows by considering
a complete $(\chi (H)-1)$-partite graph~$G$ whose vertex classes are as equal as
possible.
The proof of the upper bound is similar to Case~1 of the proof of
Lemma~\ref{exceptional} (see Section~\ref{sec5}). Details can be found
in~\cite{treglownphd}.%

We will not use Proposition~\ref{minHx} in the proof of Theorem~\ref{main},
but we have included it as it helps to explain the difference between
Theorems~\ref{kuhn}
and~\ref{main}. Indeed,
Theorem~\ref{thmKomlos} and Proposition~\ref{minHx} show that the minimum degree
which ensures an
almost perfect $H$-packing is larger than the minimum degree
which guarantees a copy of $H$ covering any given vertex.
In contrast, Theorems~\ref{orekomlos} and~\ref{oreHx} imply that for some~$H$
this is not true in the Ore-type case. So it is natural that $\delta_{\rm Ore}(H,n)$ 
involves this property
explicitly (since the property that every vertex is contained in a copy of $H$ is
clearly necessary
to ensure a perfect $H$-packing). 
In fact, this is the only real difference to the expression for $\delta(H,n)$ in
Theorem~\ref{kuhn}:
note that we have $\chi_{\rm Ore}(H)=\max \{ \chi^*(H),\chi'_{\rm Ore}(H) \}$
and thus Theorems~\ref{kuhn},~\ref{main} and~\ref{oreHx} imply that
$$
\delta_{\rm Ore}(H,n)=\max \left\{ 2\delta (H,n),\delta'_{\rm Ore}(H,n) \right\}+o(n).
$$
\subsection{Forcing a single copy of $H$}
In view of Theorem~\ref{oreHx}, one might also wonder what Ore-type degree condition
ensures
at least one copy of $H$ (i.e.~we do not require every vertex to lie in a copy of $H$).
It is easy to see that if 
$G$ is of order~$n$ then the condition is similar to the condition on the minimum
degree.
\begin{prop} \label{average}
For every graph $H$ and every $\eta >0$ there exists an integer
$n_0=n_0 (H, \eta)$ such that if $n \geq n_0$ and $G$ is a graph on $n$ vertices
which satisfies
$$
d(x)+d(y) \geq 2 \left( 1-\frac{1}{\chi(H)-1} +\eta \right) n
$$ 
for all non-adjacent $x \not =y \in V(G)$, then $G$ contains a copy of~$H$.
\end{prop}
Proposition~\ref{average} 
immediately follows from the Erd\H{o}s-Stone theorem and the following observation
(which we expect to be known, but we were unable to find a reference):
\begin{prop} \label{av2}
Let $G$ be a graph with $d(x)+d(y) \geq 2k$ for all 
non-adjacent $x \not =y \in V(G)$. Then $G$ has average degree at least~$k$.
\end{prop}
To prove Proposition~\ref{av2}, let $A$ be the set of vertices in $G$ whose degree
is less than~$k$
and let $B$ be the set of remaining vertices. Let $\overline{G}$ denote the
complement of $G$
and let $F$ denote the bipartite subgraph of $\overline{G}$ induced by $A$ and $B$.
Hall's theorem implies that $F$ has a matching covering all of $A$ (Hall's condition
can be
verified by noting that for all $X \subseteq A$ the number of edges in $F$ between
$X$ and the 
neighbourhood of $X$ is at least $|X|(n-k-1)$ and at most $|N(X)|(n-k-1)$).
Now apply the Ore-type degree condition to all pairs of vertices of $G$ which are
contained in this matching.

\subsection{Other structures}
As mentioned earlier, packing and embedding results in graphs of large minimum degree
have also been studied for other structures. It would be interesting to obtain Ore-type analogues
for some of these: e.g.~for the P\'osa-Seymour conjecture
which states that every graph $G$ on $n$ vertices with
$\delta (G)\geq \frac{r}{r+1} n$ contains the $r$th power of a Hamilton cycle. \cite{kky}
contains a discussion of other Ore-type results.

\section{Notation and extremal examples}\label{sec2}
Throughout this paper we omit floors and ceilings whenever this does not affect the
argument.
We write $e(G)$ to denote the number of edges of a graph~$G$, $|G|$ for its order,
$\delta (G)$
and $\Delta (G)$ for its minimum and maximum degrees respectively and $\chi (G)$ for
its chromatic number.

Given disjoint $A,B \subseteq V(G)$, an \emph{$A$-$B$ edge} is an edge of $G$ with one
endvertex in~$A$ and the other in $B$. The number of these edges is denoted by $e_G
(A,B)$ or
$e (A,B)$ if this is unambiguous. We write $(A,B)_G$ for the bipartite subgraph of~$G$
with vertex classes $A$ and $B$ whose edges are precisely the $A$-$B$ edges in $G$.

Let us now prove the lower bound in Theorem~\ref{main}.
The next proposition deals with the case when $CE (H) = \infty$.

\begin{prop}\label{prop1}
Let $H$ be a graph with $CE (H) = \infty$. Let $n\ge |H|$.
Then there exists a graph $G$ of order $n$ with 
$$d(x)+d(y) \geq 2\left(1-\frac{1}{\chi(H)}\right)n-2$$ for all non-adjacent $x\not
=y \in V(G)$
containing a vertex that does not belong to a copy of~$H$. (In particular, $G$ has no
perfect $H$-packing.)
\end{prop}
\begin{proof} Let $r:=\chi (H)$. Consider the complete $r$-partite graph of
order~$n$ whose vertex 
classes $V'_1 , V'_2,V_3,\dots,V_r$ have sizes as equal as possible, where
$|V'_1|\le |V'_2|\le |V_3|\le \dots\le |V_r|$. Note that%
      \COMMENT{Indeed, $n=|V'_1|+|V'_2|+|V_3|+\dots+|V_r|\ge
\frac{r}{2}(|V'_1|+|V'_2|)$.}
$n-|V'_1|-|V'_2|\ge n-2n/r$.

Let~$G$ be obtained from this graph by moving
all but one vertex, $w$ say, from~$V'_1$ to~$V'_2$, by making the set $V_2\supseteq
V'_2$ thus
obtained from~$V'_2$ into a clique and by deleting all the edges between~$w$ and the
vertices
in~$V_2$.

Any vertex $y \in V_3\cup \dots\cup V_r$ satisfies  $d(y) \ge n-\lceil\frac{n}{r}
\rceil\geq (1-1/\chi (H))n-1$. Thus $d(y_1)+d(y_2) \ge 2(1-1/\chi(H))n-2$
for all non-adjacent $y_1 \not =y_2 \in V(G) \backslash (\{w\} \cup V_2)$.
Moreover, $d(w)=n-|V'_1|-|V'_2|\ge n-2n/r$ and for any $z \in V_2$ we have $d(z)=n-2$.
So $d(w)+d(z) \geq 2(1-1/\chi(H))n-2$. Hence~$G$ satisfies our Ore-type degree
condition.

The neighbourhood of $w$ in $G$ induces an $(r-2)$-partite subgraph of $G$.
Therefore, since $\chi (H[N(x)])=r-1$ for all $x \in V(H)$, $w$ cannot 
play the role of any vertex in $H$. So~$G$ does not contain a copy of~$H$ covering~$w$.
\end{proof}

The following proposition will be used for the case when~$H$ is non-bipartite and
$CE (H)<\infty$.

\begin{prop} \label{prop2}
Let $H$ be a graph with $r:=\chi (H)\geq 3$ for which $m:=CE (H)<\infty$. Then there 
are infinitely many graphs $G$ whose order $n$ is divisible by $|H|$ and such that 
$$d(x)+d(y) \geq 2\left(1-\frac{1}{r-\frac{2}{m+2}}\right)n-1$$ for all non-adjacent
$x\not =y \in V(G)$ containing a vertex that does not belong to a copy of~$H$.
(In particular, $G$ has no perfect $H$-packing.)
\end{prop}
\begin{proof}
Let $t \in \mathbb N$ be such that $((m+2)r-2)(r-2)$ divides $t$. Define
$s:=2|H|/((m+2)r-2)$.
Let $G'$ be the complete $(r+m-1)$-partite graph with one vertex class $V_1$ of size
$st-1$,
$m$ vertex classes $V_{2},\dots,V_{m+1}$ of size $st$ and $r-2$ vertex classes
$V_{m+2}, \dots ,V_{r+m-1}$ of size $\frac{|H|t-(m+1)st}{r-2}$.
Let $G$ be obtained from $G'$ by adding a vertex $w$ to $G'$ such that $w$
is adjacent to precisely those vertices in $V_{m+2} \cup \dots \cup V_{r+m-1}$. So
$|G|=|H|t$.

Any $y \in V_1 \cup \dots \cup V_{m+1}$ satisfies
$$d(y)+d(w) \geq 2|H|t-(m+2)st-1=2\left(1-\frac{m+2}{(m+2)r-2}\right)|G|-1.$$
Furthermore, given any $y_1 \not = y_2 \in V_i$ for some $m+2 \le i \le r+m-1$, we have
\begin{align*}
d(y_1)+d(y_2)& =2|H|t-2\left( \frac{|H|t-(m+1)st}{r-2} \right)=
2|G|-\frac{2}{r-2}\left( 1- \frac{2(m+1)}{(m+2)r-2}\right)|G|
\\ & =2|G|- \frac{2}{r-2} \frac{(m+2)(r-2)}{(m+2)r-2} |G|=2\left(1-
\frac{m+2}{(m+2)r-2}\right)|G|.
\end{align*}
Since $d(y)+d(y') \geq d(y)+d(w)$ for any $y \not =y' \in V_i$ with $1 \leq i \leq m+1$
this implies that $G$ satisfies our Ore-type degree condition.

Suppose that~$w$ belongs to some copy $H_w$ of $H$ in~$G$. Since $\chi(G)=m+r-1$,
an optimal colouring of $G$ induces an $(m+r-1)$-colouring of $H_w$ and an
$(r-2)$-colouring of
$G[N(w)]$. But then $w$ must be playing the role of a vertex $x \in V(H)$ such that
$\chi (H[N(x)])\leq r-2$, contradicting the definition of $m=CE(H)$.
\end{proof}

We will now use Propositions~\ref{prop1} and~\ref{prop2} to prove the lower bound of
Theorem~\ref{main}.

{\removelastskip\penalty55\medskip\noindent{\bf Proof of Theorem~\ref{main} (lower
bound). }}
In the case when ${\rm hcf} (H) \not =1$ the lower bound follows from the lower
bound in Theorem~\ref{kuhn}. Proposition~\ref{prop1} settles the case when $CE (H) =
\infty$.
So we may assume that ${\rm hcf} (H)  =1$ and $CE (H) <\infty$.
In this case, the lower bound in Theorem~\ref{kuhn} also implies that
\begin{equation}\label{eq:lowerchicr}
\delta_{\rm Ore} (H,n) \geq 2(1-1/\chi _{cr} (H))n-2
\end{equation}
(for any graph $H$). Suppose first that $H$ is bipartite. Since $CE (H) <\infty$ this
means that $H$ must have an isolated vertex and so $CE (H)=0$. Thus
$\chi _{\rm Ore} (H)= \chi _{cr} (H)$ and so we are done by~(\ref{eq:lowerchicr}).

So suppose next that $\chi(H)\ge 3$. In this case the proof of
Proposition~\ref{prop2} implies
the lower bound whenever~$n$ is divisible by  $((m+2)r-2)(r-2)|H|$.
To deduce the lower bound for any~$n\ge ((m+2)r-2)(r-2)|H|$ which is divisible by $|H|$
we proceed as follows. Let $n'$ be the largest integer such that $n' \leq n$ and
$n'$ is divisible by 
$((m+2)r-2)(r-2)|H|$. Construct a graph $G$ of order $n'$ as in the proof of
Proposition~\ref{prop2}.
Then add $n-n'< ((m+2)r-2)(r-2)|H|$ new vertices to~$V_1$ so 
that these vertices have the same neighbourhoods as the original vertices in~$V_1$.
Then $|G|=n$
and by the same argument as in Proposition~\ref{prop2}, $G$ does not contain a
perfect $H$-packing.
Moreover, it is easy to check that 
$d(x)+d(y) \geq 2(1-1/(r-2/(m+2)))n-2|H|^4$ for all non-adjacent%
      \COMMENT{get $d(x)+d(y) \geq 2\left(1-\frac{1}{r-\frac{2}{m+2}}\right)n'-1\ge
2\left(1-\frac{1}{r-\frac{2}{m+2}}\right)(n-|H|^4+|H|)-1\ge
2\left(1-\frac{1}{r-\frac{2}{m+2}}\right)n-2|H|^4$}
$x\not =y \in V(G)$.
\endproof

\section{Some useful results}\label{sec3}
In Section~\ref{sec2} we proved the lower bound on $\delta_{\rm Ore} (H,n)$ in
Theorem~\ref{main}.
The following two results together imply the upper bound.
\begin{lemma}\label{generalH}
Let $H$ be a graph and let $\eta >0$. There exists an integer $n_0=n_0(H, \eta)$
such that
if $G$ is a graph whose order $n\geq n_0$ is divisible by $|H|$ and
$$d(x)+d(y) \geq 2 \left( 1 - \frac{1}{\chi (H)} + \eta \right) n$$ for all
non-adjacent
$x\not=y \in V(G)$ then $G$ contains a perfect $H$-packing.
\end{lemma}

\begin{lemma}\label{hcf}
Let $\eta>0$ and suppose that $H$ is a graph such that ${\rm hcf }(H)=1$ and $CE
(H)< \infty$.
There exists an integer $n_0=n_0(H, \eta)$ such that 
if $G$ is a graph whose order $n\geq n_0$ is divisible by $|H|$ and
\begin{align}\label{oreG}
d(x)+d(y) \geq \max \left\{ 2   \left( 1 - \frac{1}{\chi(H)-\frac{2}{CE (H)+2}} +
\eta \right)n,
2  \left( 1 - \frac{1}{\chi _{cr} (H)} + \eta  \right) n \right\}
\end{align}
for all non-adjacent $x\not =y \in V(G)$ then $G$ contains a perfect $H$-packing.
\end{lemma}
Note that Lemma~\ref{generalH} implies the upper bound on $\delta (H,n)$ by Alon and
Yuster
(which we mentioned in Section~\ref{sec1}). We now deduce Lemma~\ref{generalH} from
Lemma~\ref{hcf}.

{\removelastskip\penalty55\medskip\noindent{\bf Proof of Lemma~\ref{generalH}.}}
Let $h:=|H|$ and $r:=\chi(H)$. Given any $k\ge 2$, define $H^*$ to be the complete
$(r+1)$-partite graph
with one vertex class of size~$1$, one vertex class of size $hk-1$ and 
$r-1$ vertex classes of size $hk$. Let~$H'$ be obtained from $H^*$ by removing an edge
between some vertex~$y$ in a vertex class of size~$hk$ 
and the vertex in the singleton vertex class. So $\chi(H')=r+1$, $|H'|= hkr$ and
$\chi (H' [N(y)])=r-1$. Moreover, $CE (H')=0$ since~$N(y)$ lies in $r-1$ vertex
classes of~$H'$.
It is easy to see that~$H'$ contains a perfect $H$-packing and that ${\rm
hcf}(H')=1$. So
$\chi_{\rm Ore}(H')=\chi _{cr} (H') =(\chi (H')-1)\frac{|H'|}{|H'|-\sigma (H')}=  r
\frac{|H'| }{|H'|-1}$. 
In particular, we can choose~$k$ sufficiently large to guarantee that $1/\chi _{cr}
(H') \geq 1/\chi (H)- \eta/4$.

Consider any graph $G$ as in Lemma~\ref{generalH}.
Choose $a \leq kr$ such that $n-ah$ is divisible by $|H'|=hkr$.
Apply Proposition~\ref{average} to obtain $a$ disjoint copies of $H$ in $G$.
Remove these $a$ copies of $H$ from $G$ to obtain a graph $G'$ whose order is
divisible by $|H'|$ and which satisfies
$$ d_{G'} (x_1)+d_{G'}(x_2) \geq 2 \left( 1- \frac{1}{\chi (H)} +\frac{ \eta}{2}
\right)|G'|
\geq 2 \left( 1- \frac{1}{\chi _{cr} (H')} +\frac{\eta}{4} \right)|G'|$$
for all non-adjacent $x_1 \not =x_2 \in V(G')$.
Apply Lemma~\ref{hcf} to find a perfect $H'$-packing in~$G'$. In 
particular, this induces a perfect $H$-packing in~$G'$. Thus, together with all
those copies of~$H$ in $G-G'$
we have chosen before, we obtain a perfect $H$-packing in~$G$. 
\endproof

Thus to prove Theorem~\ref{main} it remains to prove Lemma~\ref{hcf},
which we will do in Section~\ref{sec6}.
In order to deal with the `exceptional' vertices in the proof of Lemma~\ref{hcf}
we use the following result which
implies that every vertex~$w$ of a graph~$G$ as in Lemma~\ref{hcf} is contained in a
copy of~$H$. 
We prove Lemma~\ref{exceptional} in Section~\ref{sec5}.

\begin{lemma}\label{exceptional}
Let $H$ be a graph such that $ m:=CE (H)< \infty$. Let $r:=\chi (H)$ and 
$\eta >0$. There exists an integer $n_0=n_0 (\eta , H)$ such that whenever~$G$ is
a graph on $n \ge n_0$ vertices with
\begin{align} \label{orecon}
d(x)+d(y) \ge 2\left(1- \frac{1}{r-\frac{2}{m+2}} + \eta \right) n
\end{align}
for all non-adjacent $x \not = y \in V(G)$
then every vertex of $G$ lies in a copy of~$H$ in~$G$.
\end{lemma}

The above results also imply Theorem~\ref{oreHx}:

{\removelastskip\penalty55\medskip\noindent{\bf Proof of Theorem~\ref{oreHx}.}}
The lower bound in the case when $CE (H)=\infty$ follows from Proposition~\ref{prop1}.
If $CE (H)<\infty$ and $\chi(H)\ge 3$ then Proposition~\ref{prop2} gives the lower
bound
for infinitely many values of~$n$ and as in the proof of the lower bound in
Theorem~\ref{main}
it can be used to derive the lower bound for any~$n$. If $CE (H)<\infty$ and
$\chi(H)=2$
then $CE (H)=0$ and so the lower bound is trivial.
The upper bound follows%
      \COMMENT{Lemma~\ref{generalH} can also be used to deal with the case when~$n$ is
not divisible by $|H|$}
from Lemmas~\ref{generalH} and~\ref{exceptional}.
\endproof

In our proof of Lemma~\ref{hcf} we will also use the following result, Lemma~12 from~\cite{kuhn2}.
It gives a sufficient condition on the sizes of the vertex classes of a
complete $\chi (H)$-partite graph~$G$ which ensures that~$G$ has a perfect $H$-packing.
Lemma~\ref{useful} is the point where the assumption that~${\rm hcf} (H)=1$ is
crucial --
it is false for graphs with ${\rm hcf} (H) \neq 1$.

\begin{lemma}\label{useful}
Let $H$ be a graph with ${\rm hcf} (H)=1$. Put $r:= \chi (H)$ and
$\gamma :=(r-1) \sigma (H)/(|H|- \sigma (H))$. Let
$0< \beta _1 \ll \lambda _1 \ll \gamma, 1- \gamma , 1/|H|$ be positive constants.
Suppose that $G$ is a complete $r$-partite graph with vertex classes $U_1, \dots, U_r$
such that $|G| \gg|H|$ is divisible by $|H|$,
$(1- \lambda _1 ^{1/10})|U_r| \leq \gamma |U_i|\leq (1- \lambda _1)|U_r|$ for all
$i<r$ and such that $|\,|U_i|-|U_j| \,|\leq \beta _1 |G|$
whenever $1 \le i<j<r$. Then $G$ contains a perfect $H$-packing.
\end{lemma}
Here (and later on) we write $0<a_1 \ll a_2 \ll a_3 \leq 1$ to mean that we can
choose the constants
$a_1,a_2,a_3$ from right to left. More
precisely, there are increasing functions $f$ and $g$ such that, given
$a_3$, whenever we choose some $a_2 \leq f(a_3)$ and $a_1 \leq g(a_2)$, all
calculations needed in the proof of Lemma~\ref{useful} are valid.


\section{The Regularity lemma and the Blow-up lemma}\label{sec4}
In the proof of Lemma~\ref{hcf} we will use Szemer\'edi's Regularity
lemma~\cite{reglem}
and the Blow-up lemma of Koml\'os, S\'ark\"ozy and 
Szemer\'edi~\cite{kss2}. 
In this section we will introduce all the information we require about these two
results.
To do this, we firstly introduce some more notation.
The \emph{density} of a bipartite graph $G$ with vertex classes~$A$ and~$B$ is
defined to be
$$d(A,B):=\frac{e(A,B)}{|A||B|}.$$
Given any $\eps>0$, we say that $G$ is \emph{$\eps$-regular} if for all sets
$X \subseteq A$ and $Y \subseteq B$ with $|X|\ge \eps |A|$ and
$|Y|\ge \eps |B|$ we have $|d(A,B)-d(X,Y)|< \eps$. 
In this case we also say that $(A,B)$ is an \emph{$\eps$-regular pair}.
Given $d \in [0,1)$ we say that $G$ is \emph{$(\eps,d)$-super-regular} if all sets
$X \subseteq A$ and $Y \subseteq B$ with $|X|\ge \eps |A|$ and $|Y|\ge \eps |B|$
satisfy $d(X,Y)>d$ and, furthermore, if $d_G (a)>d|B|$ for
all $a \in A$ and $d_G (b)>d|A|$ for all $b \in B$.

We will use the following degree form of Szemer\'edi's Regularity lemma which can be
easily derived from the classical version.
\begin{lemma}[Regularity lemma]\label{rl}
For every $\eps >0$ and each integer~$\ell_0$ there is an $M=M(\eps, \ell_0 )$ such
that if~$G$
is any graph on at least $M$ vertices and $d \in [0,1)$, then there exists a
partition of $V(G)$ 
into $\ell+1$ classes $V_0 , V_1 ,..., V_\ell$, and a spanning subgraph
$G' \subseteq G$ with the following properties:
\begin{itemize}
\item $\ell_0 \leq \ell\leq M,$ $|V_{0}| \leq \eps |G|$, $|V_1|=\dots =|V_\ell|=:L$,
\item $d_{G'} (v) > d_{G} (v) - (d+ \eps)|G|$ for all $v \in V(G)$,
\item $e(G'[V_{i}])=0$ for all $i \geq 1$,
\item for all $1 \leq i<j \leq \ell$ the graph $(V_i,V_j)_{G'}$ is $\eps$-regular
and has density
either $0$ or greater than $d$. 
\end{itemize}
\end{lemma}
The sets $V_1,\dots,V_\ell$ are called \emph{clusters}, $V_0$ is called the
\emph{exceptional set}
and the vertices in $V_0$ \emph{exceptional vertices}. We refer to $G'$ as the
\emph{pure graph of $G$}.
Clearly, we may assume that $(V_i,V_j)_{G}$ is not $\eps$-regular or has density at
most~$d$
whenever $(V_i,V_j)_{G'}$ contains no edges (for all $1\le i<j\le \ell$).
The {\emph{reduced graph $R$ of $G$} is the graph whose
vertices are $V_1, \dots ,V_\ell$ and in which $V_i$ is adjacent to $V_j$ whenever
$(V_i, V_j)_{G'}$ is $\eps$-regular and has density greater than~$d$.

A well-known fact is that the minimum degree of a graph~$G$ is almost inherited by its
reduced graph. We now prove an analogue of this for an Ore-type degree condition.
\begin{lemma}\label{inherit}
Given a constant $c$, let $G$ be a graph such that $d_{G}(x)+d_{G}(y) \geq c |G|$
for all non-adjacent $x \not =y \in V(G)$. Suppose we have applied Lemma~\ref{rl} with
parameters~$\eps$ and~$d$ to~$G$. Let~$R$ be the corresponding reduced graph.
Then $d_{R}(V_i )+d_{R} (V_j ) >(c -2d-4\eps)|R|$ for all non-adjacent $V_i \not =
V_j \in V(R)$. 
\end{lemma}
\proof
Let $V_1 ,\dots ,V_\ell$ denote the clusters obtained from Lemma~\ref{rl}. Let
$L:=|V_1|=\dots=|V_\ell|$,
let~$V_0$ denote the exceptional set and let~$G'$ be the pure graph.
Set $G'':=G'-V_0$. Consider any pair $V_iV_j$ of clusters which does not form an
edge in~$R$. Pick $x \in V_i$ and $y \in V_j$ such that%
      \COMMENT{exist by our convention about~$G'$}
$xy \not \in E(G)$. So $d_G (x)+d_G (y) \geq c|G|$ and thus $d_{G''} (x) +d_{G''}(y)
>(c-2d-4 \eps)|G|$.
However, by definition of~$G''$, each cluster containing a neighbour of~$x$ in~$G''$
must be
a neighbour of~$V_i$ in~$R$ and the analogue holds for the clusters containing the
neighbours of~$y$.
Thus $d_R (V_i)+d_R (V_j) \ge (d_{G''} (x) +d_{G''}(y))/L\ge (c- 2d -4 \eps)|R|$,
as required.
\endproof

We will also use the following Embedding lemma. The proof is based on a simple greedy
argument, see e.g.~Lemma~7.5.2 in~\cite{dies} or
Theorem~2.1 in~\cite{k} for details.

\begin{lemma}[Embedding lemma]\label{emblemma}
Let $H$ be an $r$-partite graph with vertex classes $X_1,\dots,X_r$ and let
$\eps,d,n_0$ be constants such that
$0<1/n_0\ll\eps\ll d,1/|H|$. Let~$G$ be an $r$-partite graph with vertex classes
$V_1,\dots,V_r$
of size at least~$n_0$ such that $(V_i,V_j)_{G}$ is $\eps$-regular and has density
at least~$d$
whenever~$H$ contains an edge between~$X_i$ and $X_j$ (for all $1\le i<j\le r$).
Then~$G$ contains a copy of~$H$ such that $X_i\subseteq V_i$.
\end{lemma}

The Blow-up lemma of Koml\'os, S\'ark\"ozy and Szemer\'edi~\cite{kss2} states
that one can even find a spanning subgraph~$H$ in~$G$ provided that~$H$ has bounded
maximum
degree and the bipartite pairs forming~$G$ are super-regular.

\begin{lemma}[Blow-up lemma]\label{blow}
Given a graph $R$ with $V(R)= \{1, \dots , r\}$ and $d, \Delta >0$, there is a constant
$\eps _0= \eps _0 (d, \Delta , r) >0$ such that the following holds.
Given $L_1 , \dots , L_r \in \mathbb N$ and $0 < \eps \leq \eps _0$, let $R^*$ be
the graph 
obtained from $R$ by replacing each vertex $i \in V(R)$ with a set $V_i$ of $L_i$
new vertices
and joining all vertices in $V_i$ to all vertices in 
$V_j$ precisely when $ij \in E(R)$. Let $G$ be a spanning subgraph of $R^*$ such that
for every $ij \in E(R)$ the bipartite graph $(V_i , V_j)_G$ is 
$(\eps , d)$-super-regular. Then $G$ contains a copy of every subgraph $H$ of $R^*$
with
$\Delta (H) \leq \Delta$.
\end{lemma}

\section{Sketch proof of Theorem~\ref{orekomlos}}\label{sketch}
Let $G$ be a sufficiently large graph on $n$ vertices so that
$$d(x)+d(y) \geq 2 \left( 1 - \frac{1}{\chi _{cr} (H) } \right)n$$
for all non-adjacent $x \not = y\in V(G)$. Let $r:=\chi (H)$. Denote by $B$ the complete $r$-partite graph with 
one vertex class of size $(r-1)\sigma (H)$ and $r-1$ vertex classes of size $|H|-\sigma (H)$. Note that
$\chi _{cr} (B)=\chi _{cr} (H)$ and $B$ has a perfect $H$-packing. Thus by considering $B$ instead of~$H$
if necessary, it is sufficient to prove the theorem under the added assumption that
$H$ is a complete $r$-partite graph with one vertex class of size $\sigma \in \mathbb N$ and
$r-1$ vertex classes of size $\omega \in \mathbb N$.
It suffices to consider the case when $\sigma < \omega$. (It is easy to deduce the case $\sigma=\omega$
from this using the same trick as in the proof of Lemma~\ref{generalH}.) Let $H'$ denote the complete $r$-partite graph with one
vertex class of size $\sigma$ and $r-1$ vertex classes of size $\omega -1$. 

The proof of Theorem~\ref{orekomlos} involves repeated applications of the following claim to reduced graphs of $G$.
\begin{claim}\label{tilingclaim}
Let $0<\tau,1/\ell_0\ll d'\ll \gamma,1/|H|$. Let $R'$ be a graph on $\ell'\geq \ell_0$ vertices such
that $d(x)+d(y) \geq 2 ( 1 - {1}/{\chi _{cr} (H)}-d'  )\ell'$
for all non-adjacent $x \not = y\in V(R')$. Suppose that the maximum number of vertices in $R'$ covered
by an $H$-packing is $N \leq (1- \gamma)\ell'$. Then $R'$ contains a collection of vertex-disjoint copies of
$H,$ $ H'$ and $K_r$ which together cover at least $N+\tau \ell'$ vertices.
\end{claim}
The proof of Claim~\ref{tilingclaim} is almost identical to that of Lemma~15 from~\cite{ko}. Full
details can be found in~\cite{treglown}.%
\COMMENT {claim is easily derived from Lemma 6.7 from my MSci project}
Here we briefly outline the proof. Let $\mathcal L$ denote the set of vertices not covered by the
largest $H$-packing in~$R'$. Since the subgraph of $R'$ induced by $\mathcal L$
must contain a small number of edges (else it will contain a copy of $H$), most vertices in $\mathcal L$
have `small' degree in this subgraph. Furthermore, all but at most $|H|-1$ of these vertices $x$ have
degree at least $(1-1/\chi _{cr} (H) -d')\ell'$ in $R'$ (otherwise we have a copy of $K_{|H|}$ and thus 
of $H$ in $\mathcal L$).
Now we proceed exactly as in Lemma~15 from~\cite{ko}. Indeed, for many such vertices $x$ we can
combine~$x$ with a suitable copy of $H$ in the packing and replace it with a copy of $H'$ and a
copy of $K_r$ containing $x$. Thus we obtain our desired collection of vertex-disjoint copies of $H$, $H'$ and $K_r$.

Consider a graph $F$ and $t \in \mathbb N$. Let $F(t)$ denote the graph obtained from $F$ by 
replacing every vertex $x\in V (F)$ by a set $U_x$ of $t$ independent vertices, and
joining each $u \in U_x$ to each $v\in U_y$ precisely when $xy$ is an edge in $F$. In other words we replace
the edges of $F$ by copies of $K_{t,t}$. We will refer to $F(t)$ as a 'blown-up' copy of
$F$. 

Define constants $$0\ll 1/\ell _0 \ll \eps \ll \eps'\ll d \ll d' \ll \gamma \ll \eta.$$
Apply the Regularity lemma with parameters $\eps$, $d$ and $\ell _0$ to $G$ to obtain clusters
$V_1, \dots , V_{\ell}$ of size $L$, an exceptional set $V_0$ and a reduced graph $R$. By Lemma~\ref{inherit}
we have that
$$d_R (V_{i}) +d_R (V_{j}) \geq 2 \left( 1- \frac{1}{\chi _{cr} (H)} -2d \right)|R|$$
for all $V_{i} \not = V_{j}$ with $V_{i}V_{j} \not \in E (R)$.

We wish to find an $H$-packing either in $R$ or in a blown-up copy of $R$, which covers at least
a $(1-\gamma)$-proportion of the clusters. Let $N$ denote the maximum number of clusters in $R$ covered by
an $H$-packing. If $N\geq (1-\gamma)\ell$ we are done. If not, then by 
Claim~\ref{tilingclaim} $R$ contains a collection of vertex-disjoint copies of $H$, $H'$ and $K_r$
which together cover at least $N+\tau \ell$ vertices. Let $$t:= (\omega-\sigma) |H|.$$
It is not hard to see that $H(t)$, $H'(t)$ and $K_r (t)$ all have perfect $H$-packings. 
Thus $R_1:=R(t)$ contains an $H$-packing covering at least $(N+\tau \ell)t$ vertices. Since $|R_1|=\ell t$,
a larger proportion of the vertices in $R_1$ are covered by this $H$-packing compared to the $H$-packing in $R$.
If $(N+\tau \ell)t\geq (1-\gamma)\ell t$ we are done. Otherwise we can continue this process: By applying
Claim~\ref{tilingclaim} to $R_1$ we see that $R_1 (t)=R(t^2)$ contains an $H$-packing covering a
substantially larger proportion of the vertices than the previous $H$-packing. Eventually we obtain a
graph $R'$, where $R'=R(t^k)$ for some constant $k=k(H,\gamma,d')$, such that there is
an $H$-packing $\mathcal H$ covering at least $(1-\gamma)|R'|$ clusters in $R'$. 

Suppose that $(A,B)$ is an $\eps$-regular pair of density at least $d$. By removing at most $t^k$ vertices
from $A$ and $B$, we can partition both $A$ and $B$ into $t^k$ equal subclusters $A_1, \dots ,A_{t^k}$ and
$B_1, \dots , B_{t^k}$ respectively. We can do this in such a way that each
$(A_i, B_j)$ is an $\eps '$-regular pair with density at least $d-\eps$.
So since each edge of $R$ corresponds to an $\eps$-regular pair of density at least $d$, the edges
of $R'$ can be viewed as corresponding to $\eps '$-regular pairs with density at least $d-\eps$. 
(Thus for each $V_i \in V(R)$ there are $t^k$ vertices in $R'$ which correspond to subclusters of $V_i$.)

Suppose  that $H^*$ is a copy of  $H$ in $\mathcal H$. Consider the subgraph $H^* _G$ of $G$ whose
vertex set consists of all those vertices lying in the clusters of $H^*$, and whose edge set consists of
all those edges which lie in an $\eps '$-regular pair corresponding to an edge in $H^*$. It is easy to
see (for example by repeated use of Lemma~\ref{emblemma}) that there is an $H$-packing covering
almost all the vertices in $H^* _G$. We find such an $H$-packing for each copy of $H$ in $\mathcal H$.
Since $\mathcal H$ covers at least $(1-\gamma)|R'|$ of the clusters in $R'$, and since
$\gamma \ll \eta$, the union of all these $H$-packings covers at least $(1-\eta)n$ vertices of $G$, as desired.

\section{Proof of Lemma~\ref{exceptional}}\label{sec5}
Let~$H$ be as in the statement of the lemma and let~$G$ be a graph of sufficiently
large order~$n$
which satisfies~(\ref{orecon}). Recall that $r=\chi (H)$ and $m=CE (H)$. Let~$x$ be
any vertex of~$G$.
We have to find a copy of $H$ in $G$ which contains~$x$. Suppose first that $r=2$.
Then~$H$
must have an isolated vertex~$v$ (since $CE(H)<\infty$). So we can apply
Proposition~\ref{average} to find a copy
of $H-v$ in~$G-x$ and thus a copy of~$H$ in~$G$ (where~$x$ plays the role of~$v$).

So suppose that $r\ge 3$. 
Choose additional constants $\eps,d$ and~$\alpha$ such that
\begin{equation*}
0 < \eps \ll d \ll \alpha \ll \eta
\end{equation*}
and let $\ell_0:=1/\eps$. Apply the Regularity lemma with parameters $\eps ,d,
\ell_0$ to~$G$
to obtain clusters $V_1,\dots,V_\ell$ of size~$L$, an exceptional set~$V_0$, a pure
graph $G'$
and a reduced graph~$R$. Let 
$$
k:=(m+2)r-2.
$$
Lemma~\ref{inherit} implies that
\begin{equation}\label{eq:oreR1}
d_R (V_i)+d_R (V_j) \geq 2\left( 1- \frac{1}{r-\frac{2}{m+2}}+\frac{\eta}{2}\right)|R|
= 2\left( 1- \frac{m+2}{k}+\frac{ \eta}{2}\right)|R|
\end{equation}
for all $V_i \not = V_j \in V(R)$ with $V_i V_j \not \in E (R)$.
By adding the vertices of one cluster to~$V_0$ if necessary (and deleting this
cluster from~$R$)
we may assume that $x\in V_0$. (So now $|V_0|\le 2\eps n$.)
We say that \emph{$x$ is adjacent to a cluster $V_i \in V(R)$}
if $x$ is adjacent to at least $\alpha L$ vertices of $V_i$ in $G$. We denote by~$S$
the set
of clusters $V_i \in V(R)$ that~$x$ is adjacent to, and define $s:=|S|/|R|$.
Also, we write $\overline{S}:=V(R) \setminus S$.
Note that 
\begin{equation} \label{dgx}
d_G(x) \le |S| L +|\overline{S}| \alpha L +|V_0| \le (s+\alpha+2\eps) n \le
(s+2\alpha )n 
\end{equation}
and so
\begin{equation} \label{lowers}
s\ge \frac{\delta(G)}{n}-2\alpha
\stackrel{(\ref{orecon})}{\ge}\left( 1-\frac{2}{r-\frac{2}{m+2}}+2\eta\right)-2 \alpha \ge 1-\frac{2(m+2)}{k}+\eta.
\end{equation}
In particular $s>0$ since $r\ge 3$.
Our aim now is to find either a copy $K'_r$ of~$K_r$ in $R$ containing $r-1$
clusters adjacent to~$x$
(i.e.~$|V(K'_r) \cap S| \ge r-1$), or
a copy $K'_{r+m}$ of~$K_{r+m}$ in~$R$ containing~$r-2$ clusters adjacent to~$x$. 
In both cases we could apply the Embedding lemma (Lemma~\ref{emblemma}) to find the
desired copy $H_x$
of~$H$ in~$G$. Indeed, in the case where we find $K'_{r+m}$ we could use~$x$ to 
play the role of a vertex $y \in V(H)$ for which there exists an $(r-2)$-colouring of
$H[N(y)]$ that can be extended to an $(r+m)$-colouring of~$H$. The neighbourhood
$N_{H}(y)$
of~$y$ would be embedded into the clusters belonging to $V(K'_{r+m}) \cap S$ and
$H-N_{H}(y)$
would be embedded into the clusters belonging to $V(K'_{r+m})$
(so here we use the fact that $CE(H)=m$). In the case where we find~$K'_r$, $x$ can
play the role
of any vertex of $H$. Given some optimal colouring of $H$, the vertices of $H$ which
have a different colour
than $x$ are embedded into the clusters in $V(K'_r) \cap S$ (so we only use that
$\chi(H)=r$ in this case).

Let $C$ be the set of clusters $U\in S$ with $d_R(U) < (1-(m+2)/k+\eta/2)|R|$.
By~(\ref{eq:oreR1}), $C$ induces a clique. So we may assume that $|C|<r$, since otherwise we have our copy $K'_r$ of $K_r$. 
Suppose now that for some $1\leq i\leq r-1$ we%
      \COMMENT{clearly we may assume $i\ge 1$}
have already found~$i$ clusters $U_1, \dots,U_i \in S \setminus C$ such that 
$U_1, \dots, U_i$ form a copy~$K'_i$ of~$K_i$ in~$R$. Then
\begin{equation} \label{intersect}
|\bigcap _{1\le j\le i}  N_R (U_j)|  \ge
-(i-1)|R| + \sum_{j=1}^i d_R(U_j)  \ge \left(1-\frac{i(m+2)}{k} + \eta/2 \right)|R|. 
\end{equation}

\medskip

\noindent
{\bf Case 1.} $1-s \le (2m+2)/k$

\smallskip

\noindent
In this case, we will find a copy of $K_r$ which contains at least $r-1$ vertices in
$S$.
Suppose that $i \le r-2$ and we have found $U_1, \dots,U_i$ as above. 
Then $1-i(m+2)/k \ge (2m+2)/k$ and so~(\ref{intersect}) implies that
the common neighbourhood $N_R (K'_i)$ of $K'_i$ satisfies
$|N_R (K'_i) | \ge \left( 1- s +\eta/2 \right) |R|.$
So we can choose $U_{i+1}\in S \setminus C$ to extend $K'_i$ into a copy of
$K_{i+1}$ in $R[S \setminus C]$
(we can avoid $C$ when choosing $U_{i+1}$ since $|C| < r \ll \eta |R|$).
If $i=r-1$, then $1-\frac{i(m+2)}{k}= \frac{m}{k} \ge 0$. So $|N_R (K'_i) | \ge
\eta|R|/2$ 
and we can extend $K'_i=K'_{r-1}$ into the desired copy of $K_r$ using an arbitrary
vertex of $R$.

\medskip

\noindent
{\bf Case 2.} $1-s \ge (2m+2)/k$

\smallskip

\noindent
In this case, we will either find a copy of $K_r$ which contains at least $r-1$
vertices in $S$
or find a copy of $K_{r+m}$ which contains at least $r-2$ vertices in $S$.
Suppose that $i \le r-3$ and we have found $U_1, \dots,U_i$ as described before Case~1
which form a copy $K'_i$ of $K_i$ in $R[S \setminus C]$. 
Note that
$$
1- \frac{i(m+2)}{k} \ge \frac{k-(r-3)(m+2)}{k}=\frac{3(m+2)-2}{k} \ge 
\frac{2(m+2)}{k} 
\stackrel{(\ref{lowers})}{\ge} 1-s. 
$$
Thus~(\ref{intersect}) implies that 
we can choose a cluster $U_{i+1} \in S \setminus C$ which forms a $K_{i+1}$ together
with $K'_i$.
This shows that we can find a copy $K'_{r-2}$ of $K_{r-2}$ which lies in
$R[S\setminus C]$.
Note that~(\ref{intersect}) also implies that
the common neighbourhood $N_R (K'_{r-2})$ of $K'_{r-2}$ satisfies
\begin{equation} \label{Kin2}
|N_R (K'_{r-2}) | \ge \left( 1- \frac{(r-2)(m+2)}{k} + \frac{\eta}{2} \right) |R| 
= \left( \frac{2(m+1)}{k} +\frac{\eta}{2} \right)|R|.
\end{equation}
Now we aim to extend $K'_{r-2}$ into a copy $K'_{r+m}$ of $K_{r+m}$.
We will aim to find the additional vertices in $\overline{S}$.
Suppose for some $0\leq i\leq m+1$ we have found~$i$ clusters $W_1, \dots,W_i \in
\overline{S}$
which together with $K'_{r-2}$ form a copy~$K'_{r-2+i}$ of~$K_{r-2+i}$ in~$R$.
We will need a lower bound on $d_R(W_j)$ for all $j=1,\dots,i$.
To derive this, note that the definition of $S$ implies that
$W_j$ contains a vertex $y$ which is not adjacent to $x$ in $G$.
So~(\ref{orecon}) and~(\ref{dgx}) and the inequality in Case~2 imply that 
$$
d_G(y) \ge \left( 2 \left (1-\frac{m+2}{k}+\eta \right) - s -2\alpha \right) n 
\ge \left(1  -\frac{2}{k} +\eta \right) n
$$
and so $d_{G'}(y)\ge (1-2/k+\eta/2)n$.
But each cluster containing a neighbour of $y$ in $G'$ must be a neighbour of $W_j$
in~$R$. 
Hence 
\begin{equation} \label{sbar2}
d_R(W_j) \ge \frac{d_{G'}(y)-|V_0|}{L} \ge \left( 1 -\frac{2}{k} \right) |R|.
\end{equation}
So the common neighbourhood $N_R (K'_{r-2+i})$ of $K'_{r-2+i}$ satisfies 
\begin{equation} \label{common}
|N_R (K'_{r-2+i})|   \ge  |N_R (K'_{r-2}) | -i|R|+\sum_{j=1}^i d_R(W_j) 
\stackrel{(\ref{Kin2}),(\ref{sbar2})}{\ge } 
\left( \frac{2(m+1)}{k}  -i\frac{2}{k} +\frac{\eta}{2}\right)|R| \ge \frac{\eta|R|}{2}.
\end{equation}
So we can choose a vertex $W_{i+1} \in V(R) \setminus C$ that is a common neighbour
of the 
clusters in~$K'_{r-2+i}$. Suppose that $W_{i+1} \in S$. Then together with $K'_{r-2}$
this forms a copy $K'_{r-1}$ of $K_{r-1}$ in $R[S \setminus C]$. 
Now~(\ref{intersect}) implies that $|N_R(K'_{r-1})| \ge \left(m/k+\eta/2 \right)|R|$
and so we can extend $K'_{r-1}$ to a copy of $K_r$ with at least $r-1$ vertices in $S$.
So we may assume that $W_{i+1} \in \overline{S}$.
Continuing in this way, we obtain a copy of~$K_{r+m}$ having~$r-2$ clusters in~$S$,
as required.


\section{Proof of Lemma~\ref{hcf}}\label{sec6}
\subsection{Preliminaries and an outline of the proof}
Let $H$, $G$ and $\eta >0$ be as in Lemma~\ref{hcf} and let $r:=\chi (H)$.
Choose $t \in \mathbb N$ such that $t|H|(r-1)\geq 4r/\eta$.
Let $z_1:=t(r-1)\sigma (H)$ and $z:=t(|H|-\sigma (H))$. Put $\gamma :=z_1/z$. Note
that $0<\gamma<1$
since ${\rm hcf} (H)=1$. Define $B^*$ to be the complete $r$-partite graph with one
vertex class
of size~$z_1$ and~$r-1$ vertex classes of size~$z$. Then~$B^*$ has a perfect
$H$-packing and
$\eta |B^*|/4\geq r$. Moreover,
\begin{align}\label{neweq}
\chi_{cr} (B^*)= \chi _{cr} (H) =(r-1) \frac{|H|}{|H|-\sigma (H)}
=r-1+\frac{(r-1)\sigma (H)}{|H|-\sigma (H)}=r-1+\gamma.
\end{align}
Choose $s \in \mathbb N$ and a new constant~$\lambda$ such that
$0 <\lambda \ll \eta, \gamma, 1-\gamma$ as well as $s_1:=\gamma (1+ \lambda)s \in
\mathbb N$ and $s_1 \leq s$.
Let~$B'$ denote the complete $r$-partite graph with one vertex class of size~$s_1$
and~$r-1$
vertex classes of size~$s$. Thus,
\begin{align}\label{neweq2} 
\chi _{cr} (B')=(r-1) \frac{|B'|}{|B'|-s_1}=r-1+ \gamma(1+\lambda).
\end{align}
Note that the proportion $\gamma(1+\lambda)$ of the size of the smallest vertex
class of~$B'$
compared to the size of one of the larger classes is slightly larger than the 
corresponding proportion~$\gamma$ associated with~$B^*$. We can therefore choose~$s$ and~$\lambda$
in such a way that~$B'$ has a perfect $B^*$-packing, and thus a perfect $H$-packing.
(Indeed, the perfect $B^*$-packing would consist of `most' but not all of the copies of
$B^*$ having their smallest vertex class lying in the smallest vertex class of $B'$.)

We now give an outline for the proof of Lemma~\ref{hcf}. We first apply the
Regularity lemma to~$G$
to obtain a reduced graph~$R$. Since~$R$ almost inherits the Ore-type condition
on~$G$ we may apply
Theorem~\ref{orekomlos} to find an almost perfect $B'$-packing of~$R$. We then remove
all clusters from~$R$ that are not covered by this $B'$-packing and add the vertices
in these
clusters to the exceptional set~$V_0$.

For each exceptional vertex $x \in V_0$, we apply Lemma~\ref{exceptional} to
find a copy of~$H$ in~$G$ containing~$x$, and remove the vertices in this copy
from~$G$.
Thus some vertices in clusters in~$R$ will be removed from~$G$. The copies of~$H$
will be chosen
to be disjoint for different exceptional vertices.

Our aim is to apply the Blow-up lemma to each copy~$B'_i$ of~$B'$ in the
$B'$-packing of~$R$
in order to find an $H$-packing in~$G$ which covers all the vertices belonging to
(the modified)
clusters in~$B'_i$. Then all these $H$-packings together with all those copies
of~$H$ chosen for the
exceptional vertices would form a perfect $H$-packing in~$G$. However, to do this,
we need that
the complete $r$-partite graph~$F^*_i$ whose $j$th vertex class is the union of all the
clusters in the $j$th vertex class of~$B'_i$ has a perfect $H$-packing.
Lemma~\ref{useful} gives a condition which guarantees this.

To apply Lemma~\ref{useful} we need that~$|F^*_i|$ is divisible by~$|H|$. We will
remove a
bounded number of further copies of~$H$ from~$G$ to ensure this (see
Section~\ref{sec6sub}).
Furthermore, we require that~$F^*_i$ has~$r-1$ vertex classes of roughly the same
size, $u$ say,
and that its other vertex class is a little larger than $\gamma u$. But this
condition will be
satisfied automatically by the choice of the sizes of the vertex classes in~$B'$.
In fact, this is the reason why we chose a $B'$-packing in~$R$ rather than a
$B^*$-packing.
The above strategy is based on that in~\cite{kuhn}. However, there are additional
difficulties.

\subsection{Applying the Regularity lemma and modifying the reduced
graph}\label{applying}
We define further constants satisfying
\begin{align*}
0< \eps \ll d \ll \eta _{1} \ll \beta \ll \alpha \ll \lambda \ll \eta , \gamma , 1-
\gamma.
\end{align*}
We also choose $\eta _1$ so that
$$\eta _1 \ll \frac{1}{|B'|}.$$
Throughout the proof we assume that the order~$n$ of our graph~$G$ is sufficiently
large for
our calculations to hold. Apply the Regularity lemma with parameters $\eps$, $d$ and
$\ell_0:=1/\eps$ to obtain clusters
$V_1, \dots , V_\ell$ of size~$L$, an exceptional set~$V_0$, a pure graph~$G'$ and a
reduced graph~$R$.
Let $m:=CE (H)$. By Lemma~\ref{inherit} we 
have that
\begin{align*}
d_R (V_{j_1})+d_R (V_{j_2}) \geq \max \left\{ 2\left(1-\frac{1}{r-\frac{2}{m+2}}+
\frac{\eta}{2}\right)|R|, 
2\left(1-\frac{1}{\chi _{cr} (H)}+ \frac{\eta}{2}\right)|R|\right\}
\end{align*}
for all $V_{j_1} \not = V_{j_2} \in V(R)$ with $V_{j_1} V_{j_2} \not \in E(R)$.
Together with~(\ref{neweq}) and~(\ref{neweq2}) this implies that
\begin{align*}
d_R (V_{j_1})+d_R (V_{j_2}) \geq 2\left(1-\frac{1}{\chi _{cr} (B')}\right)|R|
\end{align*}
for all $V_{j_1} \not = V_{j_2} \in V(R)$ with $V_{j_1} V_{j_2} \not \in E(R)$.
So we can apply Theorem~\ref{orekomlos} to~$R$ to obtain a $B'$-packing covering all
but at most
$\eta _1 |R|$ vertices. We denote the copies of~$B'$ in this packing by $B'_1,
\dots,B'_{\ell'}$.
We delete all the clusters not contained in some~$B'_i$ from~$R$ and add all
vertices lying in
these clusters to~$V_0$. So $|V_0| \leq \eps n+\eta _1 n \leq 2 \eta _1 n$.
We now refer to~$R$ as this modified reduced graph. We still have that 
\begin{align}\label{oreconR}
d_R (V_{j_1})+d_R (V_{j_2}) \geq \max \left\{ 2\left(1-\frac{1}{r-\frac{2}{m+2}}+
\frac{\eta}{4}\right)|R|, 
2\left(1-\frac{1}{\chi _{cr} (H)}+ \frac{\eta}{4}\right)|R|\right\}
\end{align}
for all $V_{j_1} \not = V_{j_2} \in V(R)$ with $V_{j_1} V_{j_2} \not \in E(R)$.
Recall that by definition of $B',$ each $B'_i$ contains a perfect $B^*$-packing.
Fix such a $B^*$-packing for each $i=1, \dots,\ell'$. The union of all 
these $B^*$-packings gives us a perfect $B^*$-packing $\mathcal B^*$ in~$R$.

Given any~$B'_i$, it is easy to check that we can replace each cluster~$V_j\in
V(B'_i)$ with a
subcluster of size $L':=(1- \eps |B'|)L$ such that for each edge $V_{j_1} V_{j_2}$
of~$B'_i$
the chosen subclusters of~$V_{j_1}$ and~$V_{j_2}$ form a $(2 \eps,
d/2)$-super-regular pair in~$G'$.
We do this for each $i=1,\dots,\ell'$ and add all the vertices not belonging to our
chosen
subclusters to~$V_0$. We now refer to these
subclusters as the clusters of~$R$. Then for every edge $V_{j_1} V_{j_2}$ of~$R$  
the pair $(V_{j_1},V_{j_2})_{G'}$ is still $2 \eps$-regular and has density more
than $d/2$.
Moreover,
\begin{equation}\label{eq:V0}
|V_0|\le 2\eta_1 n+\eps|B'|n\le 3\eta_1n.
\end{equation}
We now partition each cluster~$V_j$ into a red part~$V^{red} _j$ and a blue
part~$V^{blue}_j$
where $| \, |V^{red}_j| -|V^{blue} _j| \,| \leq \eps L'$ and 
$| \,|N_G (x) \cap V^{red}_j|-|N_G (x) \cap V^{blue}_j| \,|\leq \eps L'$ for all $x
\in V(G)$.
(Consider a random partition to see that there are~$V^{red} _j$ and~$V^{blue} _j$
with these properties.)
Together all these partitions of the clusters yield a partition of $V(G)-V_0$ into
a set~$V^{red}$ of red vertices and a set~$V^{blue}$ of blue vertices. In
Section~\ref{except}
we will choose certain copies of~$H$ in~$G$ to cover the exceptional vertices
in~$V_0$, but each of these
copies will avoid the red vertices. All the vertices contained in these copies
of~$H$ will be
removed from the clusters they belong to. However, for every edge $V_{j_1}V_{j_2}$
of~$B'_i$
the modified bipartite subgraph of~$G'$ whose vertex classes are the remainders
of~$V_{j_1}$
and~$V_{j_2}$ will still be $(5\eps,d/5)$-super-regular since it still contains all
vertices
in $V^{red}_{j_1}\cup V^{red}_{j_2}$. Furthermore, all edges in~$R$ will still
correspond to
$5\eps$-regular pairs of density more than $d/5$. After 
Section~\ref{except} we will only remove a bounded number of further vertices from
the clusters,
which will not affect the super-regularity significantly.

\subsection{Incorporating the exceptional vertices}\label{except}
In this section we cover all the exceptional vertices with vertex-disjoint copies of
$H$. Let $G^{blue}$ denote the induced subgraph
of $G$ with vertex set $V^{blue} \cup V_0$.
The definition of~$V^{blue}$, (\ref{oreG}) and~(\ref{eq:V0})
together imply that
$$
d_{G^{blue}}(x)+d_{G^{blue}}(y) \geq \max \left\{ 2   \left( 1 -
\frac{1}{r-\frac{2}{m+2}} + \frac{\eta}{2} \right)|G^{blue}|,
2  \left( 1 - \frac{1}{\chi _{cr} (H)} + \frac{\eta}{2}  \right) |G^{blue}| \right\}
$$
for all non-adjacent $x\not =y \in V(G^{blue})$. Let $v_1,\dots,v_{|V_0|}$ be an
enumeration of the
exceptional vertices. Lemma~\ref{exceptional} gives us a copy~$H_{v_1}$ of~$H$
in~$G^{blue}$ covering~$v_1$.
Delete the vertices of~$H_{v_1}$ from~$G^{blue}$ and apply the lemma again to find a
copy
$H_{v_2}$ of~$H$ covering~$v_2$. We would like to continue this way. However, for
later purposes it is
convenient to be able to assume that 
from each cluster we only delete a small proportion of vertices during this process.
So before choosing the copy~$H_{v_j}$ for~$v_j$ (say), we call $B'_i$ \emph{bad} if
it contains a cluster
meeting the copies $H_{v_1},\dots,H_{v_{j-1}}$ that we have chosen before in at
least $\beta L'$ vertices.
So at most $|V_0||H|/(\beta L')\le 3\eta_1 |H|n/(\beta L')\le \eta \ell'/10$ of
the~$B'_i$ are bad.%
      \COMMENT{since $\eta_1\ll \beta,\eta,1/|B'|$}
We delete all the vertices belonging to clusters in bad $B'_i$ from~$G^{blue}$.
Since there are at most $\eta n/10\le \eta |G^{blue}|/4$ such vertices, we can still
apply Lemma~\ref{exceptional} to find~$H_{v_j}$. Thus we can cover all the
exceptional vertices.
We remove  all the vertices lying in the copies $H_{v_1},\dots,H_{v_{|V_0|}}$ of~$H$
from the
clusters they belong to (and from~$G$).

\subsection{Making the blow-up of each $B \in \mathcal B^*$ divisible by
$|H|$}\label{sec6sub}
Given a subgraph $S \subseteq R$ we write $V_G (S)$ for the set of all those
vertices of~$G$ that 
belong to a cluster in~$S$. Our aim now is to find, for each~$B'_i$ in our
$B'$-packing in~$R$,
an $H$-packing in~$G$
covering all the vertices in~$V_G (B'_i)$. Thus, taking the union of these
$H$-packings and the copies 
of~$H$ containing the vertices in $V_0$, we will obtain a perfect $H$-packing
in~$G$. If we can
ensure that the complete $r$-partite graph whose $j$th vertex class is the union of
all clusters
in the $j$th vertex class of~$B'_i$ has a perfect $H$-packing, then by the Blow-up
lemma the subgraph 
of~$G'$ corresponding to~$B'_i$ will have a perfect $H$-packing. By
Lemma~\ref{useful} the
former will turn out to be the case provided that~$|H|$ divides~$|V_G (B'_i)|$. So our
next aim is to remove a bounded number of copies of~$H$ from~$G$ to ensure that
$|V_G (B'_i)|$
is divisible by~$|H|$ for all $i =1,\dots,\ell'$. This in turn will be achieved by
ensuring that~$|H|$
divides~$|V_G (B)|$ for all $B \in \mathcal B^*$.

Consider the auxiliary graph~$F$ whose vertices are the elements of~$\mathcal B^*$
where
$B_1 , B_2 \in \mathcal B^* $ are adjacent in $F$ if~$R$ contains 
a copy of~$K_r$ with one vertex in~$B_1$ and~$r-1$ vertices in~$B_2$ or vice versa.

Suppose first that~$F$ is connected. Consider a spanning tree~$T$ of~$F$ with root
$B_0\in \mathcal B^*$,
say. If $B_1$, $B_2\in \mathcal B^*$ are adjacent in~$F$ then by the Embedding
lemma~$G$ contains a
copy of~$H$ with one vertex in~$V_G (B_1)$ and all the other vertices in~$V_G
(B_2)$, or vice versa.
(To see this, let~$K'_r$ be a copy of~$K_r$ in~$R$ with one vertex $V\in V_R(B_1)$
and all other
vertices in~$V_R(B_2)$. Choose any $V'\in V_R(B_2)$ which is adjacent to all of
$V(K'_r)\setminus \{V \}$.
Then our copy of~$H$ will have one vertex, $v$ say, in~$V$. All other vertices
of~$H$ lying in the
same colour class as~$v$ will be embedded into~$V'$ and all the remaining vertices
of~$H$ will
be embedded into $V(K'_r)\setminus \{V \}$.) In fact, we can choose~$|H|-1$ disjoint
such copies of~$H$. 
So by removing at most $|H|-1$ such copies of~$H$ we can ensure $|V_G (B_1)|$ is
divisible by~$|H|$.

We can use this observation to `shift the remainders mod~$|H|$' along~$T$ to achieve
that~$|H|$
divides~$|V_G (B)|$ for all $B \in \mathcal B^*$ as follows.
Let~$j_{max}$ be the largest distance of some $B \in \mathcal B^*$ from~$B_0$
in~$T$. Then for
all $B \in \mathcal B^*$ of distance~$j_{max}$ from~$B_0$ we can remove copies
of~$H$ as indicated
above to ensure that~$|H|$ divides~$|V_G (B)|$. We can repeat this for all those
$B \in \mathcal B^*$ of distance $j_{max}-1$ from~$B_0$ etc.~until $|V_G (B)|$ is
divisible by~$|H|$
for all $B \in \mathcal B^*$. (This follows as $\sum _{B \in \mathcal B^*} |V_G
(B)|$ is divisible by~$|H|$
since~$|G|$ is divisible by~$|H|$.)

So we may assume that~$F$ is not connected. Let~$\mathcal C$ denote the set of all
components of~$F$.
Given $C\in\mathcal C$, we denote by $V_R (C) \subseteq V(R)$ the set of all those
clusters
which belong to some $B \in \mathcal B^*$ with $B \in C$. We write $V_G (C)\subseteq
 V(G)$
for the union of all the clusters in~$V_R (C)$. We will show that we can remove a
bounded
number of copies of~$H$ from~$G$ to achieve that~$|V_G (C)|$ is divisible by~$|H|$
for all $C \in \mathcal C$. 
As in the case when~$F$ is connected, we can then `shift the remainders mod~$|H|$'
along a spanning tree of each component to make~$|V_G (B)|$ divisible by~$|H|$ for
all $B \in \mathcal B^*$.

In the case when $r=2$ this is straightforward. Indeed, in this case $H$ contains an
isolated vertex
(since $CE(H)<\infty$). So given any $C \in \mathcal C$ we can apply the Embedding
lemma to find
$|H|-1$ vertex-disjoint copies of~$H$ in~$G$ such that one vertex (playing the role
of the isolated vertex)
lies in~$V_G (C)$ and the other vertices lie in~$V_G (C')$ for some
$C'\in \mathcal C\setminus\{C\}$. By removing a suitable number of such copies we
can ensure
that~$|H|$ divides~$|V_G (C)|$. Since in the above argument we can choose
any~$C'\in \mathcal C\setminus\{C\}$ to contain the remaining vertices of our copy
of~$H$
(and since~$|G|$ is divisible by~$|H|$) we can apply this argument repeatedly to make
$|V_G (C'')|$ divisible by $|H|$ for all $C'' \in \mathcal C$.

So now we consider the case when $r \geq 3$. We need the following claim.

\begin{claim}\label{c1}
Let $C_1, C_2 \in \mathcal C$ and let $V \in V_R(C_2)$. Then 
$$|N_R(V) \cap V_R (C_1)| < \left( 1- \frac{1}{r-1+ \gamma}\right)|V_R (C_1)|.$$
\end{claim}
\proof Suppose not. Then there exists some $B \in \mathcal B^*$ such that  $B \in
C_1$ and 
\begin{align*}
|N_R (V) \cap B| \geq \left( 1- \frac{1}{r-1+\gamma}\right) |B|=
|B| -\frac{(r-1)z+z_1}{r-1 + z_1/z}=|B|-z.
\end{align*}
Hence $V$ has a neighbour in at least $r-1$ vertex classes of~$B$. So~$R$ contains a
copy of~$K_r$
with one vertex, namely~$V$, in a copy $B_0 \in \mathcal B^*$ and~$r-1$ vertices 
in~$B$. So~$B$ and~$B_0$ are adjacent in~$F$. But they lie in different components
of~$F$, a contradiction. 
\endproof

We now show that we can remove a bounded number of copies of~$H$ from~$G$ to
make~$|V_G (C)|$
divisible by~$|H|$ for \emph{some} $C \in \mathcal C$. (In particular, if~$F$
consists of
exactly two components~$C$ and~$C'$ this also ensures that $|V_G (C')|$ is divisible
by~$|H|$.)

\begin{claim}\label{c2}
There exists a component $C \in \mathcal C$
with $|V_R (C)| \leq |R|/2$ for which we can ensure that $|H|$ divides $|V_G (C)|$ by
removing at most $|H|-1$ copies of~$H$ from~$G$.
\end{claim}
\proof 
To prove the claim we will distinguish two cases.

\medskip

\noindent
{\bf{Case~1.}} \emph{There exists a component $C_1 \in \mathcal C$ with
$|V_R (C_1)| \leq |R|/2$ and such that there is a cluster $V_1 \in V_R 
(C_1)$ with $d_R (V_1) \geq (1-1/\chi _{cr} (H)+\eta/4)|R|.$}

\smallskip

\noindent
Recall that $K^-_{r+1}$ is a $K_{r+1}$ with one edge removed.
We call the two non-adjacent vertices of $K^-_{r+1}$ \emph{small}.
We say that a copy $K'$ of $K^-_{r+1}$ in $R$ is \emph{good} if 
either (i) $V(K') \cap V_R(C_1)$ consists of a small vertex of~$K'$ or
(ii) $V(K')  \setminus V_R(C_1)$ consists of a small vertex of~$K'$.
Once we have found a good $K'$, we can use the Embedding lemma to find at most
$|H|-1$ vertex-disjoint copies of~$H$
in~$G$ such that their removal from~$G$ ensures that $|V_G (C_1)|$ is divisible
by~$|H|$, as desired.
(In case (i) precisely one vertex in each of these copies of~$H$ lies in~$V_G (C_1)$
while in case
(ii) precisely $|H|-1$  vertices in each of these copies of~$H$ lies in~$V_G (C_1)$.)
So it suffices to find a good copy of $K^-_{r+1}$.

Let $S$ denote the set of neighbours of $V_1$ outside $V_R(C_1)$ in $R$.
Let $K$ be the set of vertices $V \in S$ with $d_R(V)<(1-1/\chi _{cr} (H)+\eta/4)|R|$.
By~(\ref{oreconR}), $K$ induces a clique in $R$.
If $|K| \ge r$, then we have a found a good copy of $K^-_{r+1}$ (consisting of $V_1$
and $r$ vertices of $K$).
So we may assume that $|K| <r$.

Since $r\geq 3$ we have that $d_R (V_1)\geq (1/2+\eta/4)|R|$. 
So $|S \setminus K| \ge \eta|R|/4-r>0$. Thus we can choose $V_2 \in S \setminus K$.
By~(\ref{neweq}) the number of common neighbours of~$V_1$ and~$V_2$ in~$R$ is at least
\begin{equation} \label{eq12}
\left(1-\frac{2}{r-1+\gamma}+\frac{\eta}{4}\right)|R|.
\end{equation}
We first consider the case when at least
$(1-\frac{2}{r-1+\gamma}+\frac{\eta}{4})|V(R) \setminus V_R(C_1)|$ common neighbours
of~$V_1$ and~$V_2$
lie outside~$V_R (C_1)$. 
We claim that we can find $V_3,\dots,V_r \in S \setminus K$ which form a $K_r$ with
$V_1$ and $V_2$. 
Suppose that we have found $V_3,\dots,V_i$ where $2 \le i \le r-1$.
Note that Claim~\ref{c1} and the definition of $S$ imply that
for $j \ge 2$ the number of neighbours of $V_j$ outside $V_R(C_1)$ is at least 
$(1-1/(r-1+\gamma))|V(R) \setminus V_R(C_1)|$.
Together with~(\ref{eq12}), this implies that the common neighbourhood of
$V_1,\dots,V_i$ outside 
$V_R(C_1)$ has size at least
\begin{equation} \label{lowerV}
\left( 1-\frac{i}{r-1+\gamma}+\frac{\eta}{4}\right)|V(R) \setminus V_R(C_1)| 
\ge \frac{\eta}{4}|V(R) \setminus V_R(C_1)| > r > |K|.
\end{equation}
This shows that we can find  $V_{i+1}$ and more generally $V_3,\dots,V_r$ as required.
A similar calculation as in~(\ref{lowerV}), shows that the common neighbourhood of 
$V_2,\dots,V_r$ outside $V_R(C_1)$ is non-empty and so contains some vertex $V_{r+1}$
say. 
Together with $V_1,\dots,V_r$, $V_{r+1}$ forms a good copy of $K^-_{r+1}$.

Now consider the case when at least $(1-\frac{2}{r-1+\gamma}+\frac{\eta}{4})| V_R
(C_1)|$ common
neighbours of~$V_1$ and~$V_2$ lie inside~$V_R (C_1)$. Since $\eta |V_R (C_1)| /4\ge
\eta |B^*|/4 \ge r$
we can argue as in the previous case. Indeed, this time we choose
$V_3, \dots , V_r$ inside~$V_R (C_1)$ to obtain a copy of~$K_r$ in~$R$ with one vertex,
namely~$V_2$, outside $V_R (C_1)$. We also choose  a vertex~$V_{r+1}$ inside
$V_R (C_1)$ that is adjacent to $V_1$, $V_3,\dots,V_r$. Again, $V_1,\dots,V_{r+1}$ 
form a good copy of $K^-_{r+1}$.

\medskip

\noindent
{\bf{Case~2.}} \emph{ Every component $C \in \mathcal C$ with $|V_R (C)| \leq |R|/2$
is such that
$d_R (V) < (1-1/\chi _{cr} (H)+\eta/4)|R|$ for all $V \in V_R (C)$.}

\smallskip

\noindent
Together with~(\ref{oreconR}) this implies that $V_1V_2 \in E(R)$ for all
$V_1 \in V_R (C_1)$, $V_2 \in V_R (C_2)$ where $C_1,C_2 \in \mathcal C$ are such that
$|V_R (C_1)| , |V_R (C_2)| \leq |R|/2$. But this means that there is only one component
$C' \in \mathcal C$ with $|V_R (C')|\leq |R|/2$.%
       \COMMENT{Indeed, if another such component $C^* \in \mathcal C$ existed then
there would
be all possible edges between $V_R (C')$ and $V_R (C^*)$ in $R$, and so edges
between $C'$ and
$C^*$ in $F$, a contradiction.}
So~$F$ consists of precisely two components~$C'$ and~$C''$ where $V_R (C')$ forms a
clique in~$R$
and $|V_R (C'')|> |R|/2.$

We first consider the case when $r=3$. Note that~$R$ contains an edge between $V_R
(C')$ and
$V_R (C'')$. Indeed, if not then for any $V' \in V_R (C')$ and $V'' \in V_R (C'')$
by~(\ref{oreconR}) we have that $d_R (V')+d_R (V'') \geq 2(1-1/\chi _{cr}
(H)+\eta/4)|R|>|R|$
and so there must be an edge from~$V'$ to $V_R (C'')$ or from~$V''$ to $V_R (C')$, a
contradiction.

So since $|V_R (C')|\ge |B^*| \ge r+m$ we have a copy~$K'_{r+m}$ of~$K_{r+m}$
in~$V_R (C')$
such that there is a cluster $V''\in V_R (C'')$ adjacent to one of the clusters,
$V'$ say, of~$K' _{r+m}$.
Using the definition of~$m$ and the Embedding lemma we can find at most $|H|-1$
copies of~$H$ in~$G$ 
each containing precisely one vertex in $V_G (C'')$ such that their removal ensures
that~$|H|$
divides $|V_R (C')|$ and thus also $|V_R (C'')|$. (Indeed, by definition of~$m$
there exists a vertex~$y$
of~$H$ such that $\chi(H[N(y)])=r-2=1$ and such that some 1-colouring of $N(y)$ can
be extended to
an $(r+m)$-colouring of~$H$. So in our copies of~$H$ the vertex~$y$ will lie
in~$V''$, $N(y)$
will lie in~$V'$ and the remaining vertices of $H$ will lie in $V(K'_{r+m})$.)

Now suppose that $r \geq 4$. We claim that there exists $V'' \in V_R (C'')$ which
sends at least~$r$ edges to~$V_R (C')$ in~$R$. Suppose not. Then no $V\in V_R (C'')$
is joined to all of~$V_R (C')$. Together with the definition of~$C'$
and~(\ref{oreconR})
this implies that $d_R (V) \ge (1-1/\chi _{cr} (H) + \eta /4)|R|$. But then
$|V_R (C')| < |R|/\chi_{cr} (H)$ since otherwise~$V$ is joined to $\eta |R| /4 \geq
r$ vertices in~$V_R (C')$. 
By assumption there are less than $r|V_R (C'')| < r|R|$ edges between $V_R (C')$ and
$V_R (C'')$ in~$R$.
Moreover, by (\ref{oreconR}) and since $|V_R (C')| < |R|/\chi_{cr} (H)$ every
cluster in
$V_R (C')$ sends at least $(1-3/\chi _{cr} (H)+\eta/4)|R|> \eta |R|/4$ edges%
       \COMMENT{here we use that $r\ge 4$}
to $V_R (C'')$.
So $\eta |R| |V_R (C')|/4<r|R|$. But $|V_R (C') | \geq |B^*| \geq 4r/\eta$ by
definition of~$B^*$
and so $\eta |R| |V_R (C')|/4\geq r|R|$, a contradiction.
So indeed there exists a vertex $V'' \in V_R (C'')$ sending at least~$r$ edges to
$V_R (C')$.
As before,%
      \COMMENT{ok since the clusters in~$C'$ form a clique}
we can remove at most $|H|-1$ copies of~$H$ from~$G$ to ensure that~$|H|$ divides both
$|V_R (C')|$ and $|V_R (C'')|$.
\endproof

\begin{claim}\label{c3}
We can make $|V_G (B)|$ divisible by $|H|$ for all $B \in \mathcal B^*$ by removing
at most
$|\mathcal B^*||H|$ copies of~$H$ from~$G$.
\end{claim}
\proof
Our first aim is to take out some copies of~$H$ in~$G$ to achieve that $|V_G (C)|$
is divisible by~$|H|$
for each $C \in \mathcal C$. We apply Claim~\ref{c2}
to remove at most $|H|-1$ copies of~$H$ from~$G$ to ensure that $|V_G (C_1)|$ is
divisible by~$|H|$
for some component $C_1 \in \mathcal C$ with $|V_R (C_1)|\leq |R|/2$. Next we
consider the graphs
$F_1:=F-V(C_1)$ and $R_1:= R-V_R (C_1)$ instead of~$F$ and~$R$. Claim~\ref{c1}
and~(\ref{oreconR})
together imply that
$$
d_{R_1} (V_{j_1})+d_{R_1} (V_{j_2}) \geq2\left( 1-
\frac{1}{r-1+\gamma}+\frac{\eta}{4}\right)|R_1|
$$
for all $V_{j_1}\not = V_{j_2} \in V(R_1)$ with $V_{j_1}V_{j_2} \not \in E(R_1)$.
Now suppose that $|\mathcal C|\geq 3$. Then similarly as in the
proof of Claim~\ref{c2} we can find a component $C_2 \in \mathcal C$ with $|V_R
(C_2)|\leq |R_1|/2$
and such that by removing at most $|H|-1$ copies of~$H$ from $G$ we ensure
that~$|H|$ divides $|V_G (C_2)|$.
As~$|G|$ was divisible by~$|H|$ we can continue in this fashion to achieve that
$|V_G (C)|$
is divisible by~$|H|$ for each $C \in \mathcal C$.

During this process we have to take out at most $(|\mathcal C|-1)(|H|-1)$ copies
of~$H$ in~$G$. 
Now consider each $C\in \mathcal C$ separately. By proceeding as in the connected
case for each~$C$ and
taking out at most $(|C|-1)(|H|-1)$ further copies of~$H$ in each case, we can make
$|V_G (B)|$
divisible by~$|H|$ for all $B \in \mathcal B^*$.
Hence, in total we have taken out at most
$(|\mathcal C|-1)(|H|-1)+(|\mathcal B^*|-|\mathcal C|)(|H|-1) \leq |\mathcal
B^*||H|$ copies of~$H$.
(Note that $|\mathcal B^*||H|$ is also an upper bound on the number of copies of~$H$
removed from~$G$
in the case when $r=2$.)
\endproof

\subsection{Applying the Blow-up lemma} 
We now consider all the copies $B'_1, \dots ,B'_{\ell'}$ of~$B'$ in the $B'$-packing
of~$R$,
where the vertices of~$R$ are the modified clusters (i.e.~they do not contain the
vertices
contained in the copies of $H$ removed in Sections~\ref{except} and~\ref{sec6sub}).
For each $i \leq \ell'$ let $G'_i$ denote the $r$-partite subgraph
of~$G'$ whose $j$th vertex class is the union of all the clusters lying in the $j$th
vertex class of
$B'_i$ (for $j=1, \dots,r$). In Section~\ref{sec6sub} we made $|G'_i|=|V_G (B'_i)|$
divisible by~$|H|$
for each~$i$. Moreover, in Section~\ref{except} we removed at most $\beta L'$
vertices from each cluster.
In Section~\ref{sec6sub} we removed only a bounded number of further vertices. So
altogether
we removed at most $2\beta L'$ vertices from each cluster. 
Since $\beta \ll \lambda \ll \gamma,1-\gamma$ we may apply Lemma~\ref{useful} to
conclude that the
complete $r$-partite graph whose vertex classes are the same as the vertex classes
of~$G'_i$
has a perfect $H$-packing.

We observed at the end of Section~\ref{applying} that the choice of those copies
of~$H$ removed in
Section~\ref{except} ensures that all the bipartite subgraphs corresponding to edges
of~$B'_i$
are still $(5\eps , d/5)$-super-regular. In Section~\ref{sec6sub} we only removed
a bounded number of further vertices from each cluster. So after
Section~\ref{sec6sub} the
bipartite subgraphs of~$G'_i$ are still $(6 \eps , d/6)$-super-regular. Hence, for each
$i=1, \dots , \ell'$, we may apply the Blow-up lemma to find a perfect $H$-packing
in~$G'_i$.
All these $H$-packings together with the copies of $H$ chosen previously form a
perfect $H$-packing
in~$G$, as desired.

\medskip

{\footnotesize \obeylines \parindent=0pt

Daniela K\"{u}hn, Deryk Osthus \& Andrew Treglown
School of Mathematics
University of Birmingham
Edgbaston
Birmingham
B15 2TT
UK
}

{\footnotesize \parindent=0pt

\it{E-mail addresses}:
\tt{\{kuehn,osthus,treglowa\}@maths.bham.ac.uk}}
\end{document}